\documentclass[11pt]{amsart}
\usepackage{amsthm}
\usepackage{amsmath}
\usepackage{amssymb}
\usepackage{amscd}
\usepackage{stmaryrd}
\usepackage{tikz}
\usepackage{tikz-cd}
\usepackage{lipsum}
\usepackage{adjustbox}
\usepackage{url}
\usepackage{graphicx}
\usetikzlibrary{matrix}
\usetikzlibrary{cd}
\usepackage{siunitx}
\usepackage{mathrsfs}
\usepackage{pgf,tikz}
\usepackage{mathrsfs}
\usepackage{adjustbox}
\usetikzlibrary{arrows}
\usepackage{amsmath,amsthm}
\usepackage{latexsym,amssymb}
\usepackage{amssymb}
\usepackage{amscd}
\usepackage{tikz}
\usepackage{url}
\usetikzlibrary{matrix}
\usetikzlibrary{cd}
\usepackage{siunitx}
\usepackage{amsrefs}
\DefineSimpleKey{bib}{myurl}

\usepackage[hidelinks]{hyperref}
\hypersetup{allcolors=black}

\newcommand\myurl[1]{\url{#1}}
\allowdisplaybreaks
\title{An analytic viewpoint on the Hasse principle}
\author{Vler\"e Mehmeti}
\thanks{The author was partially funded by the S\'ephora-Berrebi scholarship.}
\begin{document}

\newtheorem{thm}{Theorem}[section]
\newtheorem{lm}[thm]{Lemma}
\newtheorem{prop}[thm]{Proposition}
\newtheorem{cor}[thm]{Corollary}
\newtheorem{conj}[thm]{Conjecture}
\newtheorem*{cor*}{Corollary}
\newtheorem*{thm*}{Theorem}
\theoremstyle{definition}
\newtheorem{defn}[thm]{Definition}
\newtheorem*{defn*}{Definition}
\newtheorem{conv}[thm]{Convention}
\newtheorem{rem}[thm]{Remark}
\newtheorem*{rem*}{Remark}
\newtheorem{ex}[thm]{Example}
\newtheorem*{ex*}{Example}
\newtheorem{set}[thm]{Setting}
\newtheorem{nota}[thm]{Notation}
\newtheorem{fact}{Statement}
\newtheorem{hy}[thm]{Hypothesis}
\newtheorem{cond}[thm]{Hypothesis}  
\newtheorem{summary}{Summary}
\newtheorem*{ack*}{Acknowledgements}
\newtheorem{cons}{Construction}

\begin{abstract}
Working on non-Archimedean analytic curves, we propose a geometric approach to the study of the Hasse principle over function fields of curves defined over a complete discretely valued field. Using it, we show the Hasse principle to be verified for certain families of projective homogeneous spaces. As a consequence, we prove that said principle holds for quadratic forms and homogeneous varieties over unitary groups,  results originally shown in \cite{ctps}, \cite{wu}, and~\mbox{\cite{parsur}.} 
\end{abstract}

\maketitle


\section*{Introduction}
\subsection*{History} \sloppypar Let $K$ be a number field, and $G/K$ a semisimple simply connected linear algebraic group. By results of Kneser (\cite{kne1}, \cite{kne2}), Harder (\cite{harder1}, \cite{harder2}) and Chernousov~\mbox{(\cite{che})}, we know that the following \emph{Hasse principle} holds: a torsor over~$G$ has a $K$-rational point if and only if it has rational points over all the completions of $K.$ Later on, Harder (\cite{harder3}) proved that a Hasse principle continues to hold if~$K$ is the function field of a curve defined over a finite field, meaning it holds for all global fields. In all instances, the proofs are obtained through case-by-case considerations.

During the last couple of decades, similar questions are being studied over function fields of curves defined over more general fields. In \cite{hhk}, Harbater, Hartmann and Krashen introduced a new technique to the study of such questions: \emph{field patching}. Let $F$ be the function field of a curve defined over a complete discretely valued field. Let $G/F$ be a linear algebraic group that is a rational variety. 
Through field patching, they showed that a \emph{local-global principle} holds for projective homogeneous varieties $X$ over $G$. This means there exist larger fields $F_i, i \in I,$ such that~$X$ has an $F$-rational point if and only if it has~$F_i$-rational points for all $i.$

In terms of the Hasse principle in this setting, in \cite{ctps}, Colliot-Th\'el\`ene, Parimala and Suresh propose the following conjecture:

\begin{conj}[{\cite[Conjecture 1]{ctps}}] \label{conj1}
Let $k$ be a $p$-adic field. Let $F$ be the function field of a smooth projective geometrically integral curve defined over $k$. Let $\Omega$ denote the set of discrete (rank 1) valuations on $F$ which either extend the norm of $k$ or are trivial on $k$. Let $X/F$ be a projective homogeneous space over a connected linear algebraic group $G/F$. Then, $X(F) \neq \emptyset \iff X(F_v) \neq \emptyset \  \forall v \in \Omega$, where~$F_v$ denotes the completion of $F$ with respect to $v$.   
\end{conj}

In \cite{ctps}, the conjecture is proven in the case of quadratic forms. 
Considerable progress has been made for classical linear algebraic groups in \cite{parsur}, \cite{wu}, \cite{resur} and \cite{guhsur}. In \cite{ctps} and~\cite{parsur}, field patching is a crucial ingredient; this is the case in \cite{cthhkps} as well, where~$G$ is assumed to be defined over $k^{\circ}$-the valuation ring of $k$. The authors provide a counterexample to the conjecture when $k$ is \emph{not} a $p$-adic field.

In \cite{une}, we extend field patching to Berkovich analytic curves and generalize the results of \cite{hhk}. We show that  Conjecture~\ref{conj1} is true provided~$\Omega$ is replaced by a larger set $\overline{\Omega}$ valuations. The result does not depend on $k$ being discretely valued. We show:

\begin{thm}[{\cite[Corollary 3.18]{une}}] \label{theoremval} Let $k$ be a complete ultrametric non-trivially valued field. Let $F$ be the function field of a normal irreducible projective algebraic curve  $C$ over~$k$. Let $\overline{\Omega}$ be the set of all (rank 1) non-trivial valuations $v$ on $F$ such that $v_{|k}$ is either trivial or induces the norm on $k$. Let $X/F$ be a variety on which a rational linear algebraic group~$G/F$ acts \emph{strongly transitively}. Then $$X(F) \neq \emptyset \iff X(F_v) \neq \emptyset \ \forall v \in \overline{\Omega},$$ where $F_v$ denotes the completion of $F$ with respect to $v$. 
\end{thm}

We recall:

\begin{defn} \label{st}
The linear algebraic group $G/F$ acts \emph{strongly transitively} on the variety~$X/F$ if for any field extension $L/F$, either $X(L)=\emptyset$ or $G(L)$ acts transitively on~$X(L)$. 
\end{defn}

From \cite[Remark 3.9]{hhk}, if $X/K$ is projective homogeneous over a reductive linear algebraic group $G/K$, then $G$ acts strongly transitively on $X$. This is also the case if~$X$ is a~$G$-torsor.

\subsection*{A geometric approach}
We see from Theorem \ref{theoremval} that in order to prove Conjecture~\ref{conj1} in the case $G$ is rational it suffices to show
\begin{equation}\label{rel1} {X(F_v) \neq \emptyset} \ {\forall v \in \Omega}  \Longrightarrow
 X(F_v) \neq \emptyset \ \forall v \in \overline{\Omega}.
\end{equation} 

Using the notation of Theorem~\ref{theoremval}, let $C^{\mathrm{an}}$ be the Berkovich analytification of $C$, and~$\mathscr{M}$ the sheaf of meromorphic functions on it. From \cite[Prop.~3.6.2]{Ber90}, $F=\mathscr{M}(C^{\mathrm{an}}).$ There exists a bijection $val: C^{\mathrm{an}} \rightarrow \overline{\Omega}$ (\emph{cf.} \cite[Prop. 3.15]{une}), and the set of points on~$C^{\mathrm{an}}$ corresponding to discrete valuations (\emph{i.e.} to the points of $\Omega$) can be relatively well~described.
 In \cite[Cor.~3.17]{une}, we show that for any $x \in C^{\mathrm{an}}$, $X(\mathscr{M}_x) \neq \emptyset$ if and only if $X(F_{\emph{val}(x)}) \neq \emptyset.$
Theorem \ref{theoremval} is proved as a corollary of:
$$X(F) \neq \emptyset \iff X(\mathscr{M}_x) \neq \emptyset \ \forall x \in C^{\mathrm{an}}.$$

Let us suppose  $k$ is discretely valued.
Set $S_{\mathrm{disc}}:=\{x \in C^{\mathrm{an}} : \emph{val}(x) \ \mathrm{is \ discrete}\}.$ Relation (\ref{rel1}) is then equivalent to the following:
\begin{equation}\label{rel2} {X(\mathscr{M}_x) \neq \emptyset} \ {\forall x \in S_{\mathrm{disc}}}  \Longrightarrow
 X(\mathscr{M}_x) \neq \emptyset \ \forall x \in C^{\mathrm{an}}.
\end{equation} 
The study of relation (\ref{rel2}) provides a geometric approach to Conjecture~\ref{conj1} when $G$ is a rational variety, and is the topic of study of this manuscript. 

The set $S_{\mathrm{disc}}$ is \emph{dense} in $C^{\mathrm{an}}.$  Its points can be described topologically: we recall that a Berkovich curve has the structure of a real graph; $S_{\mathrm{disc}}$ contains all its \emph{branching points} (i.e. \emph{type 2} points) and some of its \emph{extreme points} (i.e. \emph{type 1} points).

\subsection*{Main statements}
In this manuscript we show that several groups of varieties satisfy relation (\ref{rel2}), and hence (\ref{rel1}), above. Here is one of the main results we obtain:
\begin{thm}[Corollary \ref{1.7}, Proposition \ref{1.12}] \label{teo1} Let $k$ be a complete discretely valued field. Let $C/k$ be a smooth irreducible projective algebraic curve and $C^{\mathrm{an}}$ its Berkovich analytification. Set $F:=k(C)$--the function field of $C$. Let $X/F$ be a smooth proper variety such that $X(F_v) \neq \emptyset$ for all $v \in \Omega$. 
\begin{enumerate}
\item Let $Q$ be a finite subset of $S_{\mathrm{disc}} \subset C^{\mathrm{an}}$. For $x \in Q$, there exists a neighborhood~$V_x$ of $x$ in $C^{\mathrm{an}}$ such that $X(\mathscr{M}(V_x)) \neq \emptyset.$
Set $U:=C^{\mathrm{an}} \backslash \bigcup_{x \in Q} V_x$. 
If $X$ has a \emph{smooth proper} model over the ring $\mathcal{O}^{\circ}(U):=\{f \in \mathcal{O}(U): |f(x)| \leqslant 1 \ \forall x \in U\},$ then $X(F_v) \neq \emptyset$ for all $v \in \overline{\Omega}$.
\item \begin{sloppypar}
We can construct  open \emph{virtual discs} and open \emph{virtual annuli} $B_1, B_2, \dots, B_n$ in~$C^{\mathrm{an}}$ depending on $X$ such that if $X$ has  proper smooth models over $\mathcal{O}^{\circ}(B_i), {i=1,2,\dots, n}$, then $X(F_v) \neq \emptyset$ for all $v \in \overline{\Omega}$.
\end{sloppypar}
\end{enumerate}

If, in addition, there exists a rational linear algebraic group $G/F$ acting strongly transitively on $X$, then $X(F) \neq \emptyset.$
\end{thm}
The proof of (1) is based on topological considerations of~$C^{\mathrm{an}}$, as well as the nature of the rings $\mathcal{O}^{\circ}(V)$ for certain open subsets $V.$ Part (2) is then obtained as a consequence.  

Let $\mathcal{C}$ be a regular proper model of the algebraic curve $C/k$ over the valuation ring~$k^{\circ}$ of $k$. 
In~\cite{ber}, Berkovich constructed a \emph{specialization morphism} $\pi: C^{\mathrm{an}} \rightarrow \mathcal{C}_s$, where $\mathcal{C}_s$ is the special fiber of $\mathcal{C}$. Using it and a result of Bosch (see Theorem \ref{bosh}), we can interpret Theorem \ref{teo1} over models of $C$. We prove conjecture \ref{conj1} provided $X$ has \emph{proper smooth} models over $\mathcal{O}_{\mathscr{C}, Q_i}$, $i=1,2,\dots, n$, where the $Q_i$ are closed points of a fine enough regular proper model $\mathscr{C}$ depending on $X$ (\emph{cf.} Remark \ref{ecuditshme}).
These smoothness results can also be interpreted over the residue fields of the completions of $F$ (Theorem \ref{3.6}).


Combining Theorem \ref{teo1} with a theorem of Springer, we prove Conjecture~\ref{conj1} for quadratic forms, a result originally shown in \cite{ctps}. A quadratic form defined over a field~$K$ is said to be $K$-\emph{isotropic} if it has a non-trivial zero over $K$.  We continue using the same notations as in Theorem \ref{teo1}. Let us remark here that the case of residue characteristic two remains unknown. 
\begin{thm}[Theorem \ref{4.5}]
Suppose $\mathrm{char} \ \widetilde{k} \neq 2$, where $\widetilde{k}$ denotes the residue field of~$k$.  Let $q$ be a quadratic form defined over $F$. Then $q$ is $F_v$-isotropic for all $v \in \Omega$ if and only if it is $F_v$-isotropic for all $v \in \overline{\Omega}$. Consequently, if $\dim{q}>2$, $q$ is isotropic over $F$ if and only if it is isotropic over $F_v$ for all $v \in \Omega$.
\end{thm}

More generally, if $X$ satisfies such a Springer-type theorem, by combining it with Theorem \ref{teo1}, it should amount to a proof of Conjecture \ref{conj1}. In \cite{larmour}, such a ``Springer-type result'' is shown for Hermitian forms by Larmour. One can then show Theorem~\ref{7.5}, a Hasse principle for homogeneous spaces over unitary and special unitary linear algebraic groups. This result was already shown in \cite{wu} and \cite{parsur}. We merely translate the proof to our setting.

By relying heavily on the structure of Berkovich analytic curves, something can also be said about \emph{constant varieties}.  

\begin{thm}[Theorem \ref{5.6}] Suppose $k$ is a complete non-trivially valued ultrametric field.
Let $F$ be the function field of a smooth connected projective algebraic curve~$C/k$. Let~$X$ be a variety \emph{defined over~$k$}. Suppose there exists a rational linear algebraic group~${G/F}$ acting strongly transitively on ${X \times_k F}$. Under certain conditions on $C$,  if ${X(F_v) \neq \emptyset}$ for all $v \in \Omega$, then
\begin{enumerate}
\item if the value group $|k^{\times}|$ is dense in $\mathbb{R}_{>0}$, then $X(F) \neq \emptyset$;
\item if $k$ is discretely valued, then~$X$ has a zero cycle of degree one over $F$.
\end{enumerate}
\end{thm}
The conditions on $C$ are satisfied for example by curves with semi-stable reduction over~$k$ (in the discretely valued case) and Mumford curves (in general).
As for the zero cycles, there has been an extensive study of varieties for which having a zero cycle of degree one is equivalent to having a rational point. This is in particular true for quadratic forms.

\subsection*{Structure of the manuscript}
The first section contains some preliminaries and the notation we use. In Section \ref{section1}, we work over Berkovich analytic curves, and use the structure of the rings $\mathcal{O}^{\circ}(U)$ for a well-chosen open $U$ to prove a local result on the existence of rational points. The main results here are Theorem~\ref{1.5} and Proposition~\ref{1.12}.

In Section \ref{seksioni2}, we translate the results of Section \ref{section1} over models of algebraic curves.  We prove Theorems \ref{2.5} and \ref{2.13}. In Section \ref{seksioni3}, we interpret the smoothness assumptions of the previous sections over residue fields of completions of the field~$F$; we prove Theorem~\ref{3.6}.
In Section \ref{quad}, we use the techniques of Section \ref{section1} and a theorem of Springer to prove  Conjecture \ref{conj1} for quadratic forms (Theorem~\ref{4.5}). 

In Section \ref{seksioni5}, we study Conjecture \ref{conj1} for constant varieties. We construct automorphisms of the analytic curve to prove that relation (\ref{rel2}) above is satisfied. The techniques used are of different nature from those of the previous parts. The main statement here is Theorem~\ref{5.6}. 

Finally, in Appendix \ref{appendix}, we add another example to which the techniques of Section~\ref{section1} apply: homogeneous varieties over (special) unitary groups, see Theorem~\ref{7.5}.

\subsection*{Acknowledgments}
The author is grateful to Antoine Ducros, David Harari, and J\'er\^ome Poineau for insightful discussions  during the preparation of this manuscript. Many thanks also to the referee whose suggestion  considerably simplified the proof of Theorem \ref{2.13}.

\section{Preliminaries and Conventions}

\subsection{Preliminaries}
Let $k$ be a complete discretely valued field,  $k^{\circ}$ its valuation ring, and~$\widetilde{k}$ its residue field. Given a Berkovich space $X$, we denote by $\mathcal{H}(x)$ the \emph{completed residue field} of $x \in X$ (\emph{cf.} \cite[Rem. 1.2.2]{Ber90}), obtained by completing $\kappa(x)$-the residue field of $x$. Given $f \in \mathcal{O}(X)$, we denote by $f(x)$ the image of $f$ under the natural map $\mathcal{O}(X) \rightarrow \mathcal{H}(x)$, and by $|f(x)|$ its norm in the complete ultrametric field $\mathcal{H}(x)$.

\subsubsection{The specialization morphism (\emph{cf.}  \cite[Section 1]{ber})} \label{spec} 

\

\emph{(1) The affine case.} \label{para} Let $\mathscr{X}=\text{Spec} \ A$ be a flat finite type scheme over $k^{\circ}.$ The formal completion $\widehat{\mathscr{X}}$ of $\mathscr{X}$ along its special fiber is $\text{Spf}(\widehat{A}),$ where $\widehat{A}$ is a topologically finitely presented ring over $k^{\circ}$ (\emph{cf.} \cite[pg. 541]{ber}). The \textit{analytic generic fiber} of $\widehat{\mathscr{X}},$ denoted by $\widehat{\mathscr{X}}_{\eta},$ is defined to be ${\mathcal{M}(\widehat{A} \otimes_{k^{\circ}} k)},$ where $\mathcal{M}(\cdot)$ is the \textit{Berkovich spectrum} (see \cite[1.2]{Ber90}). 
 
There exists a \textit{specialization morphism} $\pi: \widehat{\mathscr{X}}_{\eta} \rightarrow \widehat{\mathscr{X}}_s,$ where $\widehat{\mathscr{X}}_s$ is the special fiber of~$\widehat{\mathscr{X}},$ which is an \emph{anti-continuous} function, meaning the pre-image of a closed subset is open. We remark that $\widehat{\mathscr{X}}_s=\mathscr{X}_s,$ where $\mathscr{X}_s:=\mathrm{Spec}(A \otimes_{k^{\circ}} \widetilde{k})$ is the special fiber of $\mathscr{X}.$ Let us describe ~$\pi$ explicitly. There are embeddings $A \hookrightarrow \widehat{A} \hookrightarrow (\widehat{A} \otimes_{k^{\circ}} k)^{\circ},$ where $(\widehat{A} \otimes_{k^{\circ}} k)^{\circ}$ is the set of all elements ~$f$ of $\widehat{A} \otimes_{k^{\circ}} k$ for which $|f(x)| \leqslant 1$ for all $x \in \mathcal{M}(\widehat{A} \otimes_{k^{\circ}} k).$ Let $x \in \mathcal{M}(\widehat{A} \otimes_{k^{\circ}} k).$ This point determines a bounded morphism $A \rightarrow \mathcal{H}(x)^{\circ},$ which induces an application ${\varphi_x : A \otimes_{k^{\circ}} \widetilde{k} \rightarrow \widetilde{\mathcal{H}(x)}}.$ Then, $\pi(x):=\ker{\varphi_x}.$

\begin{rem} \label{berred}
In \cite[2.4]{Ber90}, Berkovich constructs a \emph{reduction map} $r : \mathcal{M}(\widehat{A} \otimes_{k^{\circ}} k) \rightarrow \mathrm{Spec}(\widetilde{\widehat{A} \otimes_{k^{\circ}} k});$ here $\widetilde{\widehat{A} \otimes_{k^{\circ}} k}:=(\widehat{A} \otimes_{k^{\circ}} k)^{\circ}/(\widehat{A} \otimes_{k^{\circ}} k)^{\circ\circ}$, where $(\widehat{A} \otimes_{k^{\circ}} k)^{\circ\circ}$ is the set of $f \in \widehat{A} \otimes_{k^{\circ}} k$ such that for $x \in \mathcal{M}(\widehat{A} \otimes_{k^{\circ}} k)$, $|f(x)|<1$.  If $A$ is a normal domain, then the canonical morphism $\phi: \mathrm{Spec}(\widetilde{\widehat{A} \otimes_{k^{\circ}} k}) \rightarrow \mathrm{Spec}(A \otimes_{k^{\circ}} \widetilde{k})$ is a bijection and $\pi=\phi \circ r$. This means that some of the properties shown for $r$ in \cite[2.4]{Ber90} remain true for $\pi$. This is shown in \cite[Rem. 4.8]{mart} by using \cite[Thm. 2.1]{mart}; See also \cite[Prop.~4.1]{une} for a proof in this setting. 
\end{rem}
\emph{(2) The proper case.}
The construction from (1) has good gluing properties. Let $\mathscr{X}$ be a finite type scheme over ~$k^{\circ},$ and $\widehat{\mathscr{X}}$ its formal completion along the special fiber. The \textit{analytic generic fiber} $\widehat{\mathscr{X}}_\eta$ of $\widehat{\mathscr{X}}$ is the $k$-analytic space we obtain by gluing the analytic generic fibers of an open affine cover of $\widehat{\mathscr{X}}.$ If $\mathscr{X}$ is proper, then $\mathscr{X}_k^{\mathrm{an}}=\widehat{\mathscr{X}}_{\eta}$, where~$\mathscr{X}_k^{\mathrm{an}}$ is the Berkovich analytification of $\mathscr{X}_k:=\mathscr{X} \times_{k^{\circ}} k$ (\emph{cf.} \cite[2.2.2]{muni}). There exists an anti-continuous \textit{specialization morphism} $\pi: \mathscr{X}_k^{\mathrm{an}} \rightarrow \mathscr{X}_s,$ where $\mathscr{X}_s$ is the special fiber of~$\mathscr{X}.$ 

\subsubsection{The Theorem of Bosch}

For a $k$-analytic space $X$, let $\mathcal{O}^{\circ}_X$ be the subsheaf of~$\mathcal{O}_X$ defined by 
\begin{equation} \label{ozero}
\mathcal{O}_X^{\circ}(U):=\{f \in \mathcal{O}_X(U): |f|_x \leqslant 1 \ \forall x \in U\}\end{equation}
for any open $U$ of $X$. 
When there is no risk of ambiguity, we will simply write~$\mathcal{O}^{\circ}$.

Let $\mathscr{C}$ be a flat normal irreducible proper $k^{\circ}$-analytic curve. Let $\mathscr{C}_k$ be its generic fiber, and  $\mathscr{C}_s$ its special fiber. Let $k(\mathscr{C}_k)$ be  the function field of $\mathscr{C}_k$. The specialization morphism constructed in Section \ref{spec} is an anti-continuous map $\pi: C \rightarrow \mathscr{C}_s$, where~$C$ denotes the Berkovich analytification of $\mathscr{C}_k$.    

\begin{thm}[{\cite[Thm. 5.8]{bo1}, \cite[Thm. 3.1]{mart}}] \label{bosh}
Let $P \in \mathscr{C}_s$ be a closed point. Then $\widehat{\mathcal{O}_{\mathscr{C}, P}}=\mathcal{O}^{\circ}(\pi^{-1}(P))$, where $\widehat{\mathcal{O}_{\mathscr{C}, P}}$ is the completion of the local ring $\mathcal{O}_{\mathscr{C}, P}$ with respect to its maximal ideal. 
\end{thm}

As $P$ is a closed point and $\pi$ is anti-continuous, $\pi^{-1}(P)$ is an open subset of $C$. For a proof of Theorem \ref{bosh} in this setting, see \cite[Prop. 4.5]{une}.

\begin{rem} \label{0.2}
From \cite[Prop. 2.4.4]{Ber90} and Remark~\ref{berred}, if $x \in \mathscr{C}_s$ is the generic point of an irreducible component of $\mathscr{C}_s$, then $\pi^{-1}(x)$ is a single type 2 point of $C$. The valuation on $k(\mathscr{C}_k)$ determined by $x$ is the same as that determined by $\pi^{-1}(x)$; \emph{c.f.} \cite[Lemma 4.3]{une}. 
\end{rem}

The following objects play a very important role in this manuscript.

\begin{defn} \label{modO} \begin{sloppypar}
Let $X/k(\mathscr{C}_k)$ be a variety.
 Let $U$ be an open of $C$ such that $k(\mathscr{C}_k) \subseteq \mathrm{Frac}(\mathcal{O}^{\circ}(U))$ (\emph{e.g.} see Lemma~\ref{1.5.1}). \emph{A model} of $X$ over $\mathcal{O}^{\circ}(U)$ is an $\mathcal{O}^{\circ}(U)$-scheme~$\mathcal{X}$ such that ${\mathcal{X} \times_{\mathcal{O}^{\circ}(U)} k(\mathscr{C}_k)=X}$. 
\end{sloppypar}
\end{defn}

\subsubsection{Vertex sets and models of curves}

Let $C/k$ be a proper normal irreducible Berkovich analytic curve. We denote by $C^{\mathrm{al}}$ the unique algebraic curve such that its Berkovich analytification is $C$ (\emph{cf.} \cite[Th\'eor\`eme 3.7.2]{duc}). 

\begin{defn} \label{virtuals}  A non-empty finite set of type 2 points of $C$ is said to be a \emph{vertex set} of~$C$ (\emph{cf,} \cite[6.3.17]{duc}). 
\end{defn}

\begin{rem} \label{psejo} Let $\mathscr{C}$ be a proper flat normal model of $C^{\mathrm{al}}$ over $k^{\circ}.$ Let $\mathrm{Gen}(\mathscr{C}_s)$ denote the set of generic points of the irreducible components of the special fiber~$\mathscr{C}_s$ of $\mathscr{C}$. By Remark~\ref{0.2}, $S_{\mathscr{C}}:=\pi^{-1}(\mathrm{Gen}(\mathscr{C}_s))$ is a vertex set of $C$, where $\pi$ is the specialization morphism ${C \rightarrow \mathscr{C}_s}$. 
\end{rem}
\begin{thm}[{\cite[Theorem 4.3]{lordan}, \cite[6.3.14]{duc}}] \label{lordan}
The map $\mathscr{C} \mapsto S_{\mathscr{C}}$ induces a bijection between the following two partially ordered sets:
\begin{enumerate}
\item the isomorphism classes of flat normal proper models of $C^{\mathrm{al}}$ over $k^{\circ}$, ordered by morphisms of models,
\item the vertex sets of $C,$ ordered by inclusion.
\end{enumerate}
Furthermore, for a flat normal proper model $\mathscr{C}$ of $C^{\mathrm{al}}$ with special fiber $\mathscr{C}_s$, the corresponding specialization morphism $\pi$ induces a bijection:
\begin{center}
  $ \displaystyle
    \begin{aligned} 
        \{\mathrm{closed} \ \mathrm{points} \ \mathrm{on} \  \mathscr{C}_s\} & \cong  \{\mathrm{connected} \ \mathrm{components} \ \mathrm{of} \ C \backslash S_{\mathscr{C}}\} \\
P & \mapsto \pi^{-1}(P) \\
    \end{aligned}
  $ 
\end{center}
{If $P \in \mathscr{C}_s$ is a closed point, then the boundary of $\pi^{-1}(P)$ consists precisely of the preimages by $\pi$ of the generic points of the irreducible components of $\mathscr{C}_s$ containing~$P$.} 
\end{thm}

\begin{ex} \label{shembulli}
Given $\mathbb{P}_{\mathbb{Q}_p}^{1, \mathrm{an}}$ and the Gauss point $\eta_{0,1}$ (a semi-norm on $\mathbb{Q}_p[T]$ defined via $\sum_{i=1}^n a_n T^n \mapsto \max_i |a_n|_p$), its corresponding model over $\mathbb{Z}_p$ is $\mathbb{P}_{\mathbb{Z}_p}^1$. To see this, we may restrict to the \emph{affine} case of $\mathbb{A}_{\mathbb{Z}_p}^1$ and proceed as in Section \ref{para}(1). The specialization morphism is a map $\pi: \mathcal{M}(\mathbb{Q}_p\{T\}) \rightarrow \mathbb{A}_{\mathbb{F}_p}^1$, where $\mathbb{Q}_p\{T\}$ is the Tate algebra (\emph{cf.} \cite[2.1]{Ber90}). From Remark~\ref{berred} and \cite[Prop.~2.4.4]{Ber90}, $\pi^{-1}(\eta)$, where $\eta$ is the generic point of $\mathbb{A}_{\mathbb{F}_p}^1$, is the Shilov boundary of $\mathcal{M}(\mathbb{Q}_p\{T\})$, which is precisely $\eta_{0,1}.$ 
\end{ex}

Thanks to Hironaka's resolution of singularities:

\begin{cor} \label{lordan1}
Let $T$ be a finite set of type 2 points of $C$. There exists a proper regular model $\mathscr{C}$ of $C^{\mathrm{al}}$ over $k^{\circ}$ such that $T \subseteq S_{\mathscr{C}}$. The same remains true when replacing ``regular'' with ``\emph{sncd}''.
\end{cor}

\begin{rem} \label{sncd} We recall that a model $\mathcal{C}$ of a curve over $k^{\circ}$ is said to be \emph{sncd} if its special fiber is a \emph{strict normal crossings divisor} (\emph{cf.} \cite[Tag 0BI9]{stacks}).  In that case, $\mathcal{C}$ is regular and its singular points are ordinary double points.
\end{rem}

\subsection{Notations and conventions}  \label{1.1}

Unless stated otherwise, we use the following notation throughout the manuscript.

\begin{enumerate}
\item Let $k$ be a complete discretely valued field, $k^{\circ}$ its valuation ring, and $\widetilde{k}$ its residue field. We fix a uniformizer $t$ of $k$.

\item Let $C/k$ be a proper normal irreducible $k$-analytic curve. Set $F=\mathscr{M}(C)$, where~$\mathscr{M}$ denotes the sheaf of \emph{meromorphic functions} on $C$ (\emph{cf.} \cite[1.7]{doktoratura}). The field $\mathscr{M}_x$ of germs of meromorphic functions on $x$ is endowed with a norm. Let~$\widehat{\mathscr{M}_x}$ be its completion.

\item Let $C^{\mathrm{al}}$ be the unique projective $k$-algebraic curve whose Berkovich analytification is $C$. It is a normal and irreducible curve. Moreover, $k(C^{\mathrm{al}})=F$ (\cite[Prop. 3.6.2]{Ber90}). Given a projective algebraic curve $Y/k$, we have $Y=(Y^{\mathrm{an}})^{\mathrm{al}}$.

\item Let $V(F)$ be the set of rank 1 non-trivial valuations $v$ on $F$ such that $v_{|k}$ either induces the norm on $k$ or is trivial. For $v \in V(F)$, let~$F_v$ be the completion of $F$ with respect to $v$. 

\item Let $X/F$ be a smooth proper variety such that $X(F_v) \neq \emptyset$ for all \emph{discrete} $v \in V(F)$. 
\end{enumerate}

We recall some known consequences of these hypotheses.

\begin{rem} \label{1.2} 
From \cite[Prop.~3.15]{une}, there exists a bijection ${\emph{val}: C \leftrightarrow V(F)}$, ${x \mapsto v_x}$, such that $\widehat{\mathscr{M}_x}=F_{v_x}$. Moreover, $\emph{val}$ induces a bijection between the \emph{rigid} points of $C$ (\emph{i.e.} the Zariski closed points of $C^{\mathrm{al}}$) and $v \in V(F)$ such that $v_{|k}$ is trivial. In this case, $v$ is the discrete valuation determined by the corresponding Zariski closed point. If $x \in C$ is of type 2, then $\emph{val}(x)$ is discrete and it extends the norm on $k$.
 Consequently, ${X(\widehat{\mathscr{M}_x}) \neq \emptyset}$ for all ${x \in C}$ either a rigid or type 2 point. This is equivalent to $X(\mathscr{M}_x) \neq \emptyset$ for all such~$x$ (\emph{cf.} \cite[Cor.~3.17]{une}). In fact, \begin{equation} \label{artin}
X(\mathscr{M}_x) \neq \emptyset \ \text{for all} \ x \in C
\iff X(F_v) \neq \emptyset \ \text{for all} \ v \in V(F).
\end{equation} 
\end{rem}

\begin{rem} \label{radii}
In this manuscript, discs and annuli can have radii in $\mathbb{R}_{>0}$.
\end{rem}

\section{A smoothness criterion over analytic curves} \label{section1}

We present here some smoothness criteria over $C$ that are sufficient for the implications~\eqref{rel1} and \eqref{rel2} to hold. We start with two auxiliary results.

\begin{lm} \label{1.3bis} 
Let $V$ be a strict affinoid domain of $C.$ If $X(\mathscr{M}(V)) \neq \emptyset$, then ${X(\mathscr{M}_x) \neq \emptyset}$ for all $x \in V$.
\end{lm}

\begin{proof}
It is clearly true for $x \in \mathrm{Int} \ V.$ As $V=\partial{V} \cup \mathrm{Int} \ V$ (see \cite[Cor. 1.8.11]{doktoratura}), it suffices to prove it for points $x \in \partial{V}$, which is a consequence of Rem. \ref{1.2}.
\end{proof}

We will now use the sheaf $\mathcal{O}^{\circ}$ (\emph{cf.} relation \eqref{ozero}). 
\begin{lm} \label{1.5.1}
Let $A$ be a connected open subset of $C$ such that $\partial{A}$ consists of type~2 and~3 points. Then $F \subseteq \mathrm{Frac} \ \mathcal{O}^{\circ}(A).$ 
\end{lm}

\begin{proof}
By assumption, $\partial{\overline{A}}$ contains only type 2 and~3 points, so $\overline{A} \neq C$. Let $y \in C \backslash \overline{A}$. For any $z \in \overline{A}$, let $U_z$ be an affinoid neighborhood of $z$ such that $y \not\in U_z$ (recall $C$ is separated). By compacity, there exists a finite number of points $z \in \overline{A}$ such that $\overline{A} \subseteq \bigcup_{z} U_z$; moreover $A \subseteq V:=\bigcup_{z} U_z \neq C$. From \cite[Thm. 1.8.15(2)]{doktoratura}, $V$ is an affinoid domain of $C$. For any $a \in F \subseteq \mathscr{M}(V)$, there exist $b, c \in \mathcal{O}(V)$ such that $a=\frac{b}{c}$. The functions $b,c$ are bounded on $V$. Hence, for a large enough $n$, $t^nb$ and $t^nc$ are bounded by~$1$ in $V$. Consequently, they are bounded by $1$ in $A$, implying $a \in \mathrm{Frac} \ \mathcal{O}^{\circ}(A).$
\end{proof}

To see why the statement is not always true, take $C=\mathbb{P}_k^{1, \mathrm{an}}$ and $A=\mathbb{A}_k^{1, \mathrm{an}}$. The meromorphic function $T$ (after a choice of coordinates) is not contained in $\mathrm{Frac}  \ \mathcal{O}^{\circ}(\mathbb{A}_k^{1, \mathrm{an}})$. 

\sloppypar
\begin{thm} \label{1.5}
Let $x\in C$. Suppose there exists a connected open neighborhood $T_x$ of~$x$ in $C$ such that $\partial{T_x}$ contains only type 2 and 3 points, and $X$ has a proper smooth model ${\mathcal{X} \rightarrow \mathrm{Spec} \ \mathcal{O}^{\circ}(T_x)}.$ Then there exists a neighborhood $U_x \subseteq T_x$ of $x$ such that ${X(\mathscr{M}(U_x)) \neq \emptyset}$. In particular, $X(\mathscr{M}_x) \neq \emptyset.$
\end{thm}

\begin{proof} By Lemma \ref{1.5.1}, $F \subseteq \mathrm{Frac} \ \mathcal{O}^{\circ}(T_x)$. Let $V_x \subseteq T_x$ a strict affinoid neighborhood of~$x$, so that $\partial{V_x}$ is a finite set of type ~2 points. We remark that $X$ has a proper smooth model over~$\mathcal{O}^{\circ}(V_x).$ 
Let~$\mathscr{C}$ be a proper regular model of $C^{\mathrm{al}}$ over $k^{\circ}$ corresponding to a vertex set~$S$ of $C$ containing~$\partial{V_x}$ (see Corollary~\ref{lordan1}). We denote by $\mathscr{C}_s$ its special fiber and by~$\pi$ the  specialization morphism $C \rightarrow \mathscr{C}_s.$ If $x \in S$, then $X(\mathscr{M}_x) \neq \emptyset$ by assumption. If $x \not \in S$, set $P_x:=\pi(x)$. This is a closed point of ~$\mathscr{C}_s$ (see Rem. \ref{psejo}). By Theorem~\ref{lordan}, $U_x:=\pi^{-1}(P_x)$ is a connected component of $C \backslash S$, implying $U_x \cap S =\emptyset$. 

We continue the proof with two auxiliary lemmas which we will also use later on.

\begin{lm} \label{ekstra}
The following is satisfied: $U_x \subseteq V_x.$
\end{lm}

\begin{proof}
Assume that there exists $y \in U_x \backslash V_x$. As $U_x$ is connected, there exists an injective path $[x,y]$ connecting $x$ and $y$ which is entirely contained in $U_x$. But as $x \in V_x$ and $y \not \in V_x$, $[x,y] \cap \partial{V_x} \neq \emptyset.$ This implies that $U_x \cap S \neq \emptyset,$ contradiction. 
\end{proof}

By Lemma \ref{ekstra}, $X$ has a proper smooth model over $\mathcal{O}^{\circ}(U_x),$ which we will continue to denote by~$\mathcal{X}.$ By Theorem~\ref{bosh}, $\widehat{\mathcal{O}_{\mathscr{C}, P_x}}=\mathcal{O}^{\circ}(\pi^{-1}(P_x))$, meaning that $\mathcal{O}^{\circ}(U_x)$ is a complete regular local ring of dimension 2.

\begin{lm} \label{1.5.2}
The following is satisfied:  $\mathcal{X}(\mathcal{O}^{\circ}(U_x)) \neq \emptyset$.
\end{lm}

\begin{proof} Let $\alpha, \beta$ be generators of the maximal ideal of $\mathcal{O}^{\circ}(U_x)$. Then the localization $\mathcal{O}^{\circ}(U_x)_{(\alpha)}$ is a discrete valuation ring with uniformizer $\alpha$ (\cite[Tag 0AFS]{stacks}). As $F \subseteq  \text{Frac} \ \mathcal{O}^{\circ}(T_x) \subseteq  \text{Frac} \ \mathcal{O}^{\circ}(U_x)$, the completion $\text{Frac} \ \widehat{\mathcal{O}^{\circ}(U_x)_{(\alpha)}}$ of $\mathrm{Frac} \ \mathcal{O}^{\circ}(U_x)_{(\alpha)}$ is a complete discretely valued field containing $F$. By Remark~\ref{1.6bis} below, it either extends the norm on $k$ or is trivial there. By assumption, $X(\text{Frac} \ \widehat{\mathcal{O}^{\circ}(U_x)_{(\alpha)}}) \neq \emptyset$, and so $\mathcal{X}(\text{Frac} \ \widehat{\mathcal{O}^{\circ}(U_x)_{(\alpha)}}) \neq \emptyset.$

By the valuative criterion of properness, $\mathcal{X}(\widehat{\mathcal{O}^{\circ}(U_x)_{(\alpha)}}) \neq \emptyset.$ By taking the residue field,  $\mathcal{X}(\mathcal{O}^{\circ}(U_x)_{(\alpha)}/(\alpha))=\mathcal{X}(\text{Frac} \ \mathcal{O}^{\circ}(U_x)/(\alpha)) \neq \emptyset.$
Again, $\mathcal{O}^{\circ}(U_x)/(\alpha)$ is a discrete valuation ring (with uniformizer $\beta$, see \cite[Tag 00NQ]{stacks}), so by applying the valuative criterion of properness, $\mathcal{X}(\mathcal{O}^{\circ}(U_x)/(\alpha)) \neq \emptyset.$
As the local ring $(\mathcal{O}^{\circ}(U_x), (\alpha, \beta))$ is complete, it is also complete with respect to the topology induced by the ideal $(\alpha)$. In particular, $(\mathcal{O}^{\circ}(U_x), (\alpha))$ is a Henselian couple. Seeing as $\mathcal{X}$ is smooth over $\mathcal{O}^{\circ}(U_x)$, by applying the Hensel lifting property (see \cite[Thm.~I.8]{gruson}), we obtain that ${\mathcal{X}(\mathcal{O}^{\circ}(U_x)) \neq \emptyset}.$
\end{proof}

By Lemma \ref{1.5.2}, $X(\mathrm{Frac} \ \mathcal{O}^{\circ}(U_x)) \neq \emptyset,$ so $X(\mathscr{M}(U_x)) \neq \emptyset$ and ${X(\mathscr{M}_x) \neq \emptyset}$.
\end{proof}

\begin{rem} \label{1.6bis}
In Lemma \ref{1.5.2}, let $\alpha, \beta$ be generators of the maximal ideal of $\mathcal{O}^{\circ}(U_x)$. The localizations $\mathcal{O}^{\circ}(U_x)_{(\alpha)}, \mathcal{O}^{\circ}(U_x)_{(\beta)}$ are discretely valued rings. As ${k^{\circ} \subseteq \mathcal{O}^{\circ}(U_x)},$ depending on whether the uniformizer $t$ of $k^{\circ}$ is in  $(\alpha)$ (resp. $(\beta)$) or not, the restriction of the discrete valuation on $\mathcal{O}^{\circ}(U_x)_{(\alpha)}$ (resp. $\mathcal{O}^{\circ}(U_x)_{(\beta)}$) to $k^{\circ}$ either induces the norm on $k$ or is trivial.
\end{rem}

\begin{rem} \label{1.6.1}
 If the hypotheses of Theorem~\ref{1.5} are satisfied, then $C \backslash T_x$ contains the rigid points on which $X$ has bad reduction. To see this, assume $s \in T_x$ is a bad reduction point of $X$. The map $\mathcal{O}^{\circ}(T_x) \rightarrow \mathcal{O}_{C,s}$ induces $\mathcal{O}^{\circ}(T_x) \rightarrow \kappa(s),$ where $\kappa(s)$ is the residue field of~$s$. But then $X$ has a smooth model over~$\kappa(s)$, contradiction.   
\end{rem}

Before giving a similar global criterion for checking that $X(\mathscr{M}_x) \neq \emptyset$, we start with an auxiliary result. 

\begin{lm} \label{1.7.1}
Let $V_i, i=1,2,\dots, n$, be strict affinoid domains of $C$. Let $A$ be a finite set of type 2 points in $C$. Set $D:=C \backslash (A \cup \bigcup_{i=1}^n V_i).$ Then $D$ is an open subset of $C$ and $\partial{D}$ is a finite set of type 2 points contained in $A \cup \bigcup_{i=1}^n \partial{V_i}$.
\end{lm}
\begin{proof}
As $C$ is separated, its points are closed, so $A$ is a closed subset, implying $A \cup \bigcup_{i=1}^n V_i$ is closed, and so $D$ is open.  As $D \cap \partial{D}=\emptyset$, we obtain $\partial{D} \subseteq \bigcup_{i=1}^n V_i \cup A$. Suppose there exists $i' \in \{1,2,\dots, n\}$ such that $\partial{D} \cap \mathrm{Int} \ V_{i'} \neq \emptyset$. Let $\eta \in \partial{D} \cap \mathrm{Int} \ V_{i'}.$ There exists a neighborhood $U_{\eta}$ of $\eta$ in ~$C$ such that $U_{\eta} \subseteq V_{i'}.$ But as $\eta \in \partial{D},$ $U_{\eta} \cap D \neq \emptyset$, so $V_{i'} \cap D\neq \emptyset$, contradiction.  From \cite[Cor. 1.8.11]{doktoratura},  $\partial{D} \subseteq \bigcup_{i=1}^n \partial{V_i} \cup A$, meaning $\partial{D}$ contains only type~2 points. As $ \bigcup_{i=1}^n \partial{V_i} \cup A$ is finite (\emph{cf.} \cite[Prop.~2.5]{une}), so is $\partial{D}$.
\end{proof}

The following is a consequence of Theorem~\ref{1.5}. See the notion of \emph{strongly transitive action} in Definition~\ref{st}. 

\begin{cor} \label{1.7} 
Let $Q \subseteq C$ be a finite set of rigid and type 2 points. For ${z \in Q}$, let $V_z$ be a strict affinoid neighborhood of $z$ such that $X(\mathscr{M}(V_z)) \neq \emptyset$. 
Set ${U:=C \backslash (\bigcup_{z \in Q} V_z \cup A)}$, with $A$ a finite set of type 2 points. If there exists a smooth proper model ${\mathcal{X} \rightarrow \mathrm{Spec} \ \mathcal{O}^{\circ}(U)}$ of~$X$,
then $X(\mathscr{M}_x) \neq \emptyset$ for all $x \in C$. Equivalently, $X(F_v) \neq \emptyset$ for all $v \in V(F)$.

Moreover, if there exists a rational linear algebraic group $G/F$ acting strongly transitively on $X$, then $X(F) \neq \emptyset$.
\end{cor}

\begin{proof} 
By Lemma \ref{1.7.1}, $U$ is an open subset of $C$ and $\partial{U}$ is a finite set of type 2 points.
For any $x \in U$, by taking the connected component of $U$ containing $x$ and applying Theorem~\ref{1.5}, $X(\mathscr{M}_x) \neq \emptyset$. 
On the other hand, for all $z\in Q$, $X(\mathscr{M}(V_z)) \neq \emptyset,$ so for any $x \in \bigcup_{z \in Q} V_z$, $X(\mathscr{M}_x) \neq \emptyset$ (see Lemma \ref{1.3bis}). Thus, $X(\mathscr{M}_x) \neq \emptyset$ for all $x \in C$, so we can conclude by relation \eqref{artin} and \cite[Cor.~3.18]{une}. 
\end{proof}

\begin{rem} \label{1.8} 
(1) One can show as in Remark \ref{1.6.1} that  $\bigcup_{z \in Q} V_z$ contains the bad-reduction points of $X$. Without loss of generality, we may assume $Q$ contains the bad-reduction points of $X$. 

(2) Since $U$ does not contain the bad reduction points of $X$, the latter has a proper smooth model over $\mathcal{O}(U)$. However, $\mathcal{O}(U)$ is in general substantially larger than its subring $\mathcal{O}^{\circ}(U)$ (\emph{e.g.} if $U$ is the open unit disc over $\mathbb{Q}_p$, then $\mathcal{O}^{\circ}(U)=\mathbb{Z}_p[[T]]$ and $\mathcal{O}(U) \supsetneq \mathbb{Z}_p[[T]][1/p]$--the \emph{bounded} functions on $U$; \emph{cf.} \cite{mart}). Hence, the assumption of smoothness over $\mathcal{O}^{\circ}(U)$ is stronger.
In the next section, we interpret it as a smoothness assumption over parts of the special fiber of a model of $C^{\mathrm{al}}$.
\end{rem}

As proven by the following remark and proposition, it suffices to show that $X$ has a smooth model over $\mathcal{O}^{\circ}(U)$ for a \emph{finite} number of open subsets $U \subseteq C$, which are considerably smaller than $V$ from Corollary \ref{1.7}.

\begin{rem} \label{1.11} With the same notation as in Corollary \ref{1.7}, let $S$ be any vertex set of~$C$ such that $\bigcup_{z \in Q} \partial{V_z} \subseteq S.$ Then for any $s \in S$, $X(\mathscr{M}_s) \neq \emptyset$, so there exists a neighborhood~$U_s$ of $s$ in $C$ such that $X(\mathscr{M}(U_s)) \neq \emptyset.$ Without loss of generality, we may assume that $U_s$ has finite boundary in $C$.
There exist only a finite number of connected components of $C \backslash S$ not entirely contained in $\bigcup_{s \in S} U_s$: if that were not the case, as $S$ is finite, there exists $s_0 \in S$ such that there are infinitely many connected components of~$C \backslash S$ not contained in $\bigcup_{s\in S} U_s$ and intersecting~$U_{s_0}$. Then, they all intersect the boundary of~$U_{s_0}$, implying the latter is infinite, contradiction. 

Let $A_1, A_2, \dots, A_n$ be the connected components of $C \backslash S$ not contained in $\bigcup_{s \in S} U_s$. If~$S$ is a \emph{triangulation} (\emph{cf.} \cite[5.1.13]{duc}), then $A_i$ is an \emph{open virtual disc} or \emph{open virtual annulus} (\emph{cf.} \cite[pg. 210]{duc}). 
We remark that $\partial{A_i} \subseteq S$. For any $i,$ set $B_i:= A_i \backslash \bigcup_{z \in Q} V_z$. This implies that $\partial{B_i} \subseteq \partial{A_i} \cup \bigcup_{z \in Q} \partial{V_z}$. Hence, $B_i$ is an open subset of $C$ and $\partial{B_i} \subseteq S \cup \bigcup_{s \in S} \partial{U_s}$ (i.e. a finite set of type 2 points). See Figure 1 for illustrations of such data. 

\

\begin{minipage}{0.4 \textwidth}
\definecolor{ccqqqq}{rgb}{0.8,0,0}
\definecolor{qqqqff}{rgb}{0,0,1}
\definecolor{qqwuqq}{rgb}{0,0.39215686274509803,0}
\begin{center}
\begin{tikzpicture}[line cap=round,line join=round,>=triangle 45,x=1cm,y=1cm, scale=0.45]
\draw [line width=0.4pt] (-1.5,2.86)-- (-1.44,-3.44);
\draw [line width=0.4pt] (-1.486858334844912,1.4801251587157627)-- (-0.82,2.9);
\draw [line width=0.4pt] (-1.486858334844912,1.4801251587157627)-- (0.32,1.94);
\draw [line width=0.4pt] (-0.5071415861119455,1.7294790098940034)-- (0,2.76);
\draw [line width=0.4pt] (-1.486858334844912,1.4801251587157627)-- (-2.96,2.74);
\draw [line width=0.4pt] (-1.486858334844912,1.4801251587157627)-- (-3.36,1.52);
\draw [line width=0.4pt] (-2.158841309048093,1.4944301175667223)-- (-2.76,1.14);
\draw [line width=0.4pt] (-1.4502884092145836,-2.35971703246871)-- (-0.1,-3.48);
\draw [line width=0.4pt] (-1.4502884092145836,-2.35971703246871)-- (1.3,-2.32);
\draw [line width=0.4pt] (-0.42021890086866476,-2.344841754672049)-- (0.56,-1.56);
\draw [line width=0.4pt] (-1.4502884092145836,-2.35971703246871)-- (-3.56,-2.72);
\draw [line width=0.4pt] (-2.3055592611295848,-2.5088814777408297)-- (-2.645559261129585,-3.3488814777408296);
\draw [line width=0.4pt] (-1.473490809573764,0.07653500524522583)-- (-0.008988215405471484,0.6967704893361524);
\draw [line width=0.4pt] (-0.8266136284352528,0.3504957228216079)-- (0.03237645206391054,-0.006428857643342467);
\draw [line width=0.4pt] (-1.4600996268663489,-1.3295391790333624)-- (-2.677009267180612,-0.5855342022146911);
\draw [line width=0.4pt] (-2.0563301829150014,-0.9650104693117407)-- (-3.132020609343814,-1.4541922190717143);
\draw [shift={(-2.6791380120464825,-2.8340117380516845)},line width=0.5pt,color=qqqqff]  plot[domain=-1.6539252309690085:1.9167833610331142,variable=\t]({0.977995007669109*2.7744996184153807*cos(\t r)+-0.20862829380095965*1.0723085125109393*sin(\t r)},{0.20862829380095965*2.7744996184153807*cos(\t r)+0.977995007669109*1.0723085125109393*sin(\t r)});
\draw [shift={(-1.9594062235702767,2.0447020385524928)},line width=0.5pt,color=qqqqff]  plot[domain=-1.4813531749605922:2.074876586527798,variable=\t]({0.06094845263380595*1.686367592091549*cos(\t r)+0.9981409149621834*1.5872996441746399*sin(\t r)},{-0.9981409149621834*1.686367592091549*cos(\t r)+0.06094845263380595*1.5872996441746399*sin(\t r)});
\draw [line width=0.6pt,color=ccqqqq] (-0.3719473659379591,2.004196141057353)-- (-0.0010922423470400844,2.757780543548919);
\draw [line width=0.6pt,color=ccqqqq] (-0.3928131990237959,1.7585774443714126)-- (0.3186326371238615,1.939651983902935);
\draw [line width=0.6pt,color=ccqqqq] (-1.4769019362960565,0.4347033110859335)-- (-1.4571687208104382,-1.6372843149039833);
\draw [line width=0.6pt,color=ccqqqq] (-1.473490809573764,0.07653500524522588)-- (-0.008988215405471456,0.6967704893361524);
\draw [line width=0.6pt,color=ccqqqq] (-0.8266136284352528,0.3504957228216079)-- (0.03237645206391049,-0.006428857643342445);
\draw [line width=0.6pt,color=ccqqqq] (-3.129376720614705,-1.4529898823371856)-- (-2.0563301829150005,-0.9650104693117406);
\draw [line width=0.6pt,color=ccqqqq] (-1.4600996268663489,-1.3295391790333622)-- (-2.677009267180612,-0.5855342022146912);
\draw [line width=0.6pt,color=ccqqqq] (-0.06350689385142438,-2.0592295539157712)-- (0.558137964900774,-1.561490894424952);
\draw [line width=0.6pt,color=ccqqqq] (0.04348498836232043,-2.3381453869883)-- (1.2960850670321329,-2.320056535714319);
\draw [line width=0.6pt,color=ccqqqq] (-0.5337460297634614,-3.120137406155351)-- (-0.1,-3.48);
\begin{scriptsize}
\draw [fill=orange] (-1.486858334844912,1.4801251587157627) circle (3pt);
\draw [fill=orange] (-1.4502884092145836,-2.3597170324687102) circle (3pt);
\end{scriptsize}
\end{tikzpicture}
\end{center}
\end{minipage}
\hfill
\begin{minipage}{0.5 \textwidth}
\definecolor{ccqqqq}{rgb}{0.8,0,0}
\definecolor{qqwuqq}{rgb}{0,0.39215686274509803,0}
\definecolor{qqqqff}{rgb}{0,0,1}
\begin{tikzpicture}[line cap=round,line join=round,>=triangle 45,x=1cm,y=1cm, scale=0.8]
\draw [line width=0.4pt] (5.105838402878936,2.0522645874389487)-- (3.5963636363636406,1.3945454545454523);
\draw [line width=0.4pt] (4.357998943965344,1.7264106338841678)-- (3.90545454545455,0.9945454545454522);
\draw [rotate around={179.7293422211855:(6.574071100844763,1.9763283485870986)},line width=0.4pt] (6.574071100844763,1.9763283485870986) ellipse (1.1516305483343348cm and 0.5134680959634843cm);
\draw [line width=0.4pt] (7.687272727272733,2.1036363636363613)-- (10.087272727272733,2.176363636363634);
\draw [line width=0.4pt] (8.56358997772579,2.132704079771983)-- (9.945408159543968,2.932704079771983);
\draw [line width=0.4pt] (9.305621351125945,2.1526772310258524)-- (10.032727272727278,1.6309090909090886);
\draw [line width=0.4pt] (6.432727272727278,1.4672727272727248)-- (5.105454545454551,0.9218181818181795);
\draw [line width=0.4pt] (5.814370320640996,1.2131534318948007)-- (5.9963636363636414,0.8309090909090886);
\draw [shift={(6.572640554724287,2.004363913094897)},line width=0.5pt,color=qqqqff]  plot[domain=-0.7784037553433398:2.2065428014777515,variable=\t]({-0.9995393206570573*0.8265306742072708*cos(\t r)+0.030350394732667648*0.3444443078272129*sin(\t r)},{-0.030350394732667648*0.8265306742072708*cos(\t r)+-0.9995393206570573*0.3444443078272129*sin(\t r)});
\draw [shift={(5.5105959925979,2.323228712177983)},line width=0.5pt,color=qqqqff]  plot[domain=2.2744046990183397:5.42529381340711,variable=\t]({-0.9792162471879514*0.6762199770621148*cos(\t r)+0.20281898639709456*0.24605572392588546*sin(\t r)},{-0.20281898639709456*0.6762199770621148*cos(\t r)+-0.9792162471879514*0.24605572392588546*sin(\t r)});
\draw [line width=0.4pt] (5.444454107279045,2.080497133485542)-- (3.475214404708215,1.9708305648400295);
\draw [line width=0.4pt] (4.456172061147784,2.02545990225783)-- (3.6058083359843316,2.5421790141730383);
\draw [shift={(6.742803101335329,1.2626671303998136)},line width=0.5pt,color=qqqqff]  plot[domain=1.3177236849151588:4.191601537582174,variable=\t]({-0.9011293146297075*1.0133546437508854*cos(\t r)+0.43355041035039227*0.3338648540008264*sin(\t r)},{-0.43355041035039227*1.0133546437508854*cos(\t r)+-0.9011293146297075*0.3338648540008264*sin(\t r)});
\draw [shift={(5.952465895532633,1.6992242578273609)},line width=0.5pt,color=qqqqff]  plot[domain=-0.9987540895340565:1.964121320213759,variable=\t]({-0.9861108860161472*1.442706937666415*cos(\t r)+0.16608829122021024*1.0207026277098625*sin(\t r)},{-0.16608829122021024*1.442706937666415*cos(\t r)+-0.9861108860161472*1.0207026277098625*sin(\t r)});
\draw [shift={(6.664530924094383,2.0228140792687173)},line width=0.5pt,color=qqqqff]  plot[domain=0.33323520742052054:1.9079786313151537,variable=\t]({0.9946148309681961*0.7671443980566389*cos(\t r)+0.10364042656273995*0.3277239386720831*sin(\t r)},{-0.10364042656273995*0.7671443980566389*cos(\t r)+0.9946148309681961*0.3277239386720831*sin(\t r)});
\draw [shift={(8.172872885774106,3.934689593297389)},line width=0.5pt,color=qqqqff]  plot[domain=4.320911181679921:5.156808190285789,variable=\t]({1*2.0342887158828526*cos(\t r)+0*2.0342887158828526*sin(\t r)},{0*2.0342887158828526*cos(\t r)+1*2.0342887158828526*sin(\t r)});
\draw [shift={(6.727966934763913,2.5332290194421336)},line width=0.5pt,color=qqqqff]  plot[domain=0.8493525249662075:3.989265977500896,variable=\t]({0.9824865243069564*0.46373149739342806*cos(\t r)+-0.18633365116166334*0.16170695190107526*sin(\t r)},{0.18633365116166334*0.46373149739342806*cos(\t r)+0.9824865243069564*0.16170695190107526*sin(\t r)});
\draw [shift={(8.981286699854026,2.2961245084978814)},line width=0.5pt,color=qqqqff]  plot[domain=-1.618077273203541:0.7605281654750143,variable=\t]({0.9206620069142507*0.4473574477440575*cos(\t r)+-0.39036069093163583*0.20866987630957382*sin(\t r)},{0.39036069093163583*0.4473574477440575*cos(\t r)+0.9206620069142507*0.20866987630957382*sin(\t r)});
\draw [shift={(8.758142927656703,2.7687335568270166)},line width=0.5pt,color=qqqqff]  plot[domain=4.498626829162853:5.456400534535516,variable=\t]({0.9871673987372206*0.7280866920770606*cos(\t r)+0.15968884391337249*0.18694604867960746*sin(\t r)},{-0.15968884391337249*0.7280866920770606*cos(\t r)+0.9871673987372206*0.18694604867960746*sin(\t r)});
\draw [shift={(6.956422633480364,2.4299563999409712)},line width=0.5pt,color=qqqqff]  plot[domain=1.0092587672493827:1.52344892533638,variable=\t]({0.9992857740283344*0.7727593745254331*cos(\t r)+0.03778811750527296*0.2818213720256205*sin(\t r)},{-0.03778811750527296*0.7727593745254331*cos(\t r)+0.9992857740283344*0.2818213720256205*sin(\t r)});
\draw [shift={(8.019699259035896,2.5855784098738783)},line width=0.5pt,color=qqqqff]  plot[domain=0.4699340251875065:2.1633332046531173,variable=\t]({0.9999731002419435*0.6242288494158861*cos(\t r)+0.0073347660164445275*0.0694918235431502*sin(\t r)},{-0.0073347660164445275*0.6242288494158861*cos(\t r)+0.9999731002419435*0.0694918235431502*sin(\t r)});
\draw [shift={(7.530317877795425,2.6667705369816592)},line width=0.5pt,color=qqqqff]  plot[domain=3.417217513176884:5.797140660852658,variable=\t]({0.9999794549395441*0.15948943509860558*cos(\t r)+-0.006410124711150102*0.04669489717628163*sin(\t r)},{0.006410124711150102*0.15948943509860558*cos(\t r)+0.9999794549395441*0.04669489717628163*sin(\t r)});
\draw [line width=0.6pt,color=ccqqqq] (4.661733930880424,2.0382880041809406)-- (3.479388461193139,1.9710630172181918);
\draw [line width=0.6pt,color=ccqqqq] (4.456172061147785,2.0254599022578303)-- (3.6058083359843316,2.5421790141730383);
\draw [line width=0.6pt,color=ccqqqq] (4.555818330406867,1.8126059106824717)-- (3.5963636363636406,1.3945454545454523);
\draw [line width=0.6pt,color=ccqqqq] (4.357998943965344,1.7264106338841674)-- (3.9092472521250574,1.0006791067500185);
\draw [line width=0.6pt,color=ccqqqq] (6.173002602008496,2.459173708889619)-- (6.273347829767765,2.4731235020676308);
\draw [shift={(6.574071100844763,1.9763283485870986)},line width=0.5pt,color=ccqqqq]  plot[domain=2.6970032045167347:3.101650575773221,variable=\t]({-0.999988842549758*1.1516305483343348*cos(\t r)+-0.004723851817655574*0.5134680959634843*sin(\t r)},{0.004723851817655574*1.1516305483343348*cos(\t r)+-0.999988842549758*0.5134680959634843*sin(\t r)});
\draw [line width=0.6pt,color=ccqqqq] (9.280333263036576,2.5476607186361226)-- (9.944357428549976,2.932095761828093);
\draw [line width=0.6pt,color=ccqqqq] (9.153009175579102,2.148052619645645)-- (10.087272727272733,2.176363636363634);
\draw [line width=0.6pt,color=ccqqqq] (9.305621351125948,2.152677231025853)-- (10.03220628305985,1.631282950895259);
\draw [line width=0.6pt,color=ccqqqq] (7.696002156975899,1.8562417566440226)-- (8.888125419557477,1.4979650308337773);
\draw [line width=0.6pt,color=ccqqqq] (8.272978842718722,1.6828391192446543)-- (8.650438415096108,1.251789204784519);
\draw [line width=0.6pt,color=ccqqqq] (8.424484883967729,1.5098227993530913)-- (8.332107605549632,1.205100686051039);
\begin{scriptsize}
\draw [fill=orange] (5.444454107279045,2.080497133485542) circle (2pt);
\draw [fill=orange] (6.433075016755525,1.469058401775782) circle (2pt);
\draw [fill=orange] (7.687272727272733,2.1036363636363613) circle (2pt);
\end{scriptsize}
\end{tikzpicture}
\end{minipage}
$$\mathrm{Figure \ 1.} \ {\color{orange}S}, {\color{blue}U_s}, s\in S, {\color{red}A_i}$$

\begin{prop} \label{1.12}
If $X$ has proper smooth models over the rings $\mathcal{O}^{\circ}(B_i)$, $i=1,2,\dots, n$,  then $X(\mathscr{M}_x) \neq \emptyset$ for all $x \in C$. Equivalently, $X(F_v) \neq \emptyset$ for all $v \in V(F).$

If, moreover, there exists a rational linear algebraic group $G/F$ acting strongly transitively on $X$, then $X(F) \neq \emptyset$.
\end{prop}

\begin{proof} If $x \in \bigcup_{s \in S} U_s \cup \bigcup_{z \in Q} V_z$, then $X(\mathscr{M}_x) \neq \emptyset.$ Assume $x \in \bigcup_{i=1}^n B_i$, and let $i_0 \in \{1,2,\dots, n\}$ be such that $z \in B_{i_0}$.
Let $\mathscr{C}$ be a proper regular model of $C^{\mathrm{al}}$ over $k^{\circ}$ corresponding to a vertex set $S'$ such that $S \cup \bigcup_{s \in S} \partial{U_s} \subseteq S'$ (see Corollary \ref{lordan1}). In particular, $\bigcup_{i=1}^n \partial{B_i} \subseteq S'$. We denote by $\mathscr{C}_s$ its special fiber and by $\pi$ the specialization morphism $C \rightarrow \mathscr{C}_s$. Set $P_x:=\pi(x)$ and $U_x:=\pi^{-1}(P_x).$ The proof of Lemma~\ref{ekstra} can be applied \emph{mutatis mutandis} to show that $U_x \subseteq B_{i_0}$, meaning $X$ has a proper smooth model over $\mathcal{O}^{\circ}(U_x)$. By Theorem \ref{1.5}, ${X(\mathscr{M}_x) \neq \emptyset}$, so we can conclude by relation \eqref{artin} and \cite[Cor.~3.18]{une}.
\end{proof}
\end{rem}

Let us illustrate how strong the hypotheses of Corollary \ref{1.7} or Proposition \ref{1.12} are with an example where the variety $X$ is given by a quadratic form.

\begin{ex} \label{1.8.1}
Let $q$ be a diagonal quadratic form over $F$ with (non-zero) coefficients ${a_1, a_2, \dots, a_m}.$ Let $Q$ be a finite closed subset of $C$ containing all the zeroes and poles of the $a_i$. We use the same notation as in Remark~\ref{1.11}. Fix $i \in \{1,2,\dots, n\}.$  

By multiplying with a high enough power of the uniformizer $t \in k$ (this does not change the Weil divisors of $a_1, a_2, \dots, a_m$ in $C$), we may assume that ${a_1, a_2, \dots, a_m \in \mathcal{O}^{\circ}(B_i)}$. Moreover, we may assume ${t^{-2}a_j \not \in \mathcal{O}^{\circ}(B_i)}$ for all $j \in \{1, 2, \dots, m\}$.

The quadratic form $q$ is smooth over $\mathcal{O}^{\circ}(B_i)$ if and only if for any maximal ideal~$M$ of $\mathcal{O}^{\circ}(B_i)$ and any $j \in \{1,2,\dots, m\}$, $a_j \not \in M$. Consequently, $a_j \in \mathcal{O}^{\circ}(B_i)^{\times}$ for all such $j$. This implies that for any $x \in B_i$ and any $j$, $|a_j|_x=1.$
\end{ex}

In Section \ref{quad}, by using the techniques of this section and a theorem of Springer, we prove a general result for quadratic forms (provided $\mathrm{char} \ \widetilde{k} \neq 2$).

\begin{rem} \label{1.8.2} An argument similar to that of Example \ref{1.8.1} can be applied to a general variety $X$. 
As the smoothness of $X$ is checked by the non-vanishing of minors $\epsilon$ of a certain matrix, it suffices to show that for any $i \in \{1, 2, \dots, n\}$, there exists a model of $X$ over $\mathcal{O}^{\circ}(B_i)$ with corresponding minors $\epsilon_i$ satisfying $|\epsilon_i|_x=1$ for all $x \in B_i$.
\end{rem}

\begin{ex}\label{1.8.3}
Let $k=\mathbb{Q}_p$, $p \neq 2,$ $C=\mathbb{P}_{\mathbb{Q}_p}^{1, \mathrm{an}}$, and $F=\mathbb{Q}_p(T)$. Set $q:=X_1^2-(1+pT)X_2^2+TX_3^2 - (T+p)X_4^2$. The bad reduction points in $\mathbb{P}_{\mathbb{Q}_p}^{1, \mathrm{an}}$ of the corresponding quadric are those for which $|T|=0$, $|T|=\infty$, $|T+p|=0$ and $|1+pT|=0$.  

Set $X_1=\sqrt{1+pT}$, $X_2=1$, and $X_3=X_4=0$; this is a zero of $q$ defined on a neighborhood $N$ of the rigid point $|T|=0$. The open $N$ is a disk centered at $0$ and of radius the radius of convergence of the series expansion of $\sqrt{1+pT}$. The latter is strictly larger than $1$, so said zero is defined over the open unit disk $D$ (\emph{i.e.} $D \subsetneq N$). Remark that the point $|T+p|=0$ is in $D$, and that $\eta_{0,1} \in N$ (\emph{cf.} Example~\ref{shembulli}) since $|T|_{\eta_{0,1}}=1$.

Set $X_1=X_2=0$, $X_3=\sqrt{1+\frac{p}{T}}$ and $X_4=1$; it is a zero of $q$ on a neighborhood of the point $|T|=\infty$. Similarly, this solution is defined on the open disk $D_{\infty}$ centered at $\infty$ and of radius $1$. We remark that $D_{\infty}$ contains the point $|1+pT|=0$. 

To the Gauss point $\eta_{0,1} \in \mathbb{P}_{\mathbb{Q}_p}^{1, \mathrm{an}}$ corresponds the model $\mathbb{P}_{\mathbb{Z}_p}^{1, \mathrm{an}}$ (\emph{cf.} Example~\ref{shembulli}); let $\pi: \mathbb{P}_{\mathbb{Q}_p}^{1, \mathrm{an}} \rightarrow \mathbb{P}_{\mathbb{F}_p}^{1}$ be the specialization morphism.
Let $U$ be a connected component of $\mathbb{P}_{\mathbb{Q}_p}^{1, \mathrm{an}} \backslash (D \cup D_{\infty} \cup \{\eta_{0,1}\})$. Set $P=\pi(U).$ By Theorem \ref{bosh}, $\mathcal{O}^{\circ}(U)=\widehat{\mathcal{O}_{\mathbb{A}_{\mathbb{Z}_p}^{1},P}}$ is a local ring; let $m$ be its maximal ideal. As the rigid points $|T|=0$, $|T+p|=0$ and $|1+pT|=0$ are not in $U$, their images via $\pi$ are different from $P$, so the ideals $(T), (T+p)$, $(1+pT)$ are \emph{not} contained in $m$ (\emph{cf.} proof of Lemma~\ref{2.2}). Thus, the coefficients of $q$ are in $\mathcal{O}^{\circ}(U)^{\times}$, meaning that the quadric defined by $q$ is proper and smooth over $\mathcal{O}^{\circ}(U)$.

A crucial point here is that the Weil divisors in $\mathbb{P}_{\mathbb{Z}_p}^{1}$ of the coefficients of $q$ are \emph{not} vertical; see Section ~\ref{seksioni2}.
We remark that, in this particular case, as $N \cup D_{\infty}=\mathbb{P}_{\mathbb{Q}_p}^{1, \mathrm{an}}$, $q=0$ has non-trivial solutions over $\mathscr{M}_x$ for all $x \in \mathbb{P}_{\mathbb{Q}_p}^{1, \mathrm{an}}$, so it has one over $\mathbb{Q}_p(T).$
\end{ex}

\section{A smoothness criterion over fine models of algebraic curves} \label{seksioni2}
We now use Bosch's theorem to interpret over a model of the algebraic curve $C$ the results of Section \ref{section1}.  Let us start with a local statement. 

\begin{lm} \label{2.1} Let $\mathcal{C}$ be a proper regular model of $C^{\mathrm{al}}$ over $k^{\circ}.$ Let $\mathcal{C}_s$ denote its special fiber and $\pi$ the corresponding specialization morphism $C \rightarrow \mathcal{C}_s$.  
Let $P \in \mathcal{C}_s$ be a closed point. If $X$ has a smooth proper model over $\mathcal{O}_{\mathcal{C}, P}$, then $X(\mathscr{M}(\pi^{-1}(P))) \neq \emptyset.$ In particular, for any $x \in \pi^{-1}(P),$ $X(\mathscr{M}_x) \neq \emptyset.$ 
\end{lm}

\begin{proof}
By Theorem \ref{lordan}, $U:=\pi^{-1}(P)$ is a connected open subset of $C$, and its boundary contains only type 2 points. By assumption, $X$ has a proper smooth model~$\mathcal{X}$ over $\mathcal{O}^{\circ}(U)=\widehat{\mathcal{O}_{\mathcal{C}, P}}$ (Theorem \ref{bosh}). By the proof of Lemma \ref{1.5.2}, $\mathcal{X}(\mathcal{O}^{\circ}(U)) \neq \emptyset$, implying $X(\mathscr{M}(U)) \neq \emptyset$, and so for all $x \in \pi^{-1}(P)$, $X(\mathscr{M}_x) \neq \emptyset.$
\end{proof}

Let us show that if $z \in C^{\mathrm{al}}$ is a bad-reduction point of $X$, then the hypothesis of Lemma~\ref{2.1} is never satisfied for $P:=\pi(z).$  Recall that for a closed point $z \in C^{\mathrm{al}}$, the Zariski closure $\overline{\{z\}}$ of $\{z\}$ in $\mathcal{C}$ is some set $\{z, P_z\}$, where $P_z$ is a closed point of~$\mathcal{C}_s$ (see \cite[Def. 10.1.31]{liulibri}). 

\begin{lm} \label{2.2} We use the same notation as in Lemma \ref{2.1}.
Let $z \in C^{\mathrm{al}}$ be a Zariski closed point which we identify with a rigid point in $C$. Then ${\pi(z)=P_z}$. 
\end{lm}

\begin{proof}
Let $V=\mathrm{Spec} \ A$ be an open neighborhood of $P_z$ and $\pi(z)$ in $\mathcal{C}$ (\emph{cf.} \cite[2.2]{gabber}). Then $z$ corresponds to a principal prime ideal $(a)$ of $A$, so $P_z$ is the unique maximal ideal of $A$ containing both $a$ and the uniformizer $t \in k$. 
From \cite[7.1.10]{dejong}, this coincides with $\pi(z)$.
\end{proof}

\begin{rem} \label{2.3}
If $X$ has a  proper smooth  model over $\mathcal{O}_{\mathcal{C}, P}$ for some closed point ${P \in \mathcal{C}_s}$, then it has a proper smooth model over some open neighborhood $N$ of $P$ in~$\mathcal{C}$. If $z \in C^{\mathrm{al}}$ is any Zariski closed point for which $\overline{\{z\}}=\{z, P\},$ meaning $P=\pi(z)$ by Lemma~\ref{2.2}, then clearly $z \in N$, so $X$ has a proper smooth model over $\kappa(z).$ Thus,~$X$ has good reduction over $z$, which is why the bad reduction points of $X$ in $C$ will automatically be excluded in Lemma~\ref{2.1}.
\end{rem}

\begin{rem} \label{ecuditshme}
We use the notation of Remark \ref{1.11}. By Theorem \ref{lordan}, we may assume the vertex set $S$ of $C$ corresponds to a \emph{regular} proper model $\mathscr{C}$ over $k^{\circ}$ of $C^{\mathrm{al}}$. Let us denote by $\mathscr{C}_s$ its special fiber and by $\pi$ the specialization morphism $C \rightarrow \mathscr{C}_s$. By \emph{loc.cit.}, for $i \in \{1,2,\dots, n\}$, ${\pi(B_i)=\pi(A_i)=:\{Q_i\}}$, where $Q_i$ is a closed point of $\mathscr{C}$. By Theorem~\ref{bosh}, $\widehat{\mathcal{O}_{\mathscr{C}, Q_i}}=\mathcal{O}^{\circ}(A_i) \subseteq \mathcal{O}^{\circ}(B_i)$. Hence, by Proposition \ref{1.12}, if $X$ has proper smooth models over the rings $\mathcal{O}_{\mathscr{C}, Q_i}$, then $X(\mathscr{M}_x) \neq \emptyset$ for all $x \in C$, and, equivalently, $X(F_v) \neq \emptyset$ for all~${v \in V(F)}.$ See Figure 2 below for an illustration.
\end{rem}

We will now give similar, but global, versions of Lemma \ref{2.1}. In order to deal with the bad-reduction points of $X$ on $C$, we construct and work on ``fine enough" models of $C^{\mathrm{al}}$.

\begin{center}
\definecolor{ccqqqq}{rgb}{0.8,0,0}
\definecolor{qqqqff}{rgb}{0,0,1}
\definecolor{zzttff}{rgb}{0.6,0.2,1}
\definecolor{ffvvqq}{rgb}{1,0.3333333333333333,0}
\definecolor{zzttqq}{rgb}{0.6,0.2,0}
\begin{tikzpicture}[line cap=round,line join=round,>=triangle 45,x=1cm,y=1cm, scale=0.9]
\draw [line width=0.4pt] (-0.48,4.72)-- (-0.5,1.14);
\draw [line width=0.4pt] (-0.48670432557268584,3.519925722489233)-- (1.02,4.66);
\draw [line width=0.4pt] (0.3916767533585498,4.18456818547614)-- (1.3,3.86);
\draw [line width=0.4pt] (-0.4939666687472692,2.219966294238812)-- (-1.24,1.06);
\draw [shift={(-0.48801748977686904,1.172463631037081)},line width=0.4pt,color=zzttqq]  plot[domain=1.5396519666041433:4.782244576907428,variable=\t]({0.0029446125003494243*1.635448426077079*cos(\t r)+0.9999956646192136*0.8724581894977228*sin(\t r)},{-0.9999956646192136*1.635448426077079*cos(\t r)+0.0029446125003494243*0.8724581894977228*sin(\t r)});
\draw [line width=0.4pt] (-0.4952515111644949,1.9899795015554131)-- (0.16716501042297213,1.3018540097524218);
\draw [line width=0.4pt] (-0.32490499689771574,1.8130216886285284)-- (0.1958362125084269,1.8007329260393339);
\draw [line width=0.4pt] (-0.48392250820781246,4.01787103080157)-- (-1.245799103612277,4.607850161634012);
\draw [line width=0.4pt] (2.82398304652662,4.517854978952808)-- (4.293904363652954,2.517962030481606);
\draw [line width=0.4pt] (4.343901687364734,3.387915463066579)-- (2.933977158692536,1.2480300082023923);
\draw (1.5113920028598045,3.0981851252927095) node[anchor=north west] {$\xrightarrow{\pi}$};
\draw (-2.306170624043538,0.8839999266400177) node[anchor=north west] {Figure 2. ${\color{orange}{Z}}, {\color{brown}{N_z}}, z \in Z, {\color{violet}{S}}, {\color{blue}{U_s}}, s \in S, Q_i=\pi(A_i)$}; 
\draw [rotate around={89.67991517621833:(-0.48995576622048576,2.9379178465330518)},line width=0.4pt,color=qqqqff] (-0.48995576622048576,2.9379178465330518) ellipse (0.5204057489828408cm and 0.3398115734002314cm);
\draw [rotate around={89.67991517621881:(-0.4830608109331277,4.17211484297011)},line width=0.4pt,color=qqqqff] (-0.4830608109331277,4.17211484297011) ellipse (0.3773481979638565cm and 0.24949358394142745cm);
\draw [line width=0.8pt,color=ccqqqq] (-0.4851688685285512,3.794772533389337)-- (-0.4870485169007201,3.458315474771101);
\draw [line width=0.8pt,color=ccqqqq] (-0.4867043255726858,3.5199257224892326)-- (1.02,4.66);
\draw [line width=0.8pt,color=ccqqqq] (0.3916767533585488,4.184568185476139)-- (1.2977546394023909,3.860802327384662);
\draw [line width=0.8pt,color=ccqqqq] (-0.7312203014688534,4.209372565308579)-- (-1.2457991036122757,4.607850161634011);
\draw [line width=0.8pt,color=ccqqqq] (-0.48095275333770454,4.549457152550886)-- (-0.47942123842378154,4.891529290251608);
\draw [line width=0.8pt,color=ccqqqq] (-0.47926828915323627,4.669622182598323)-- (-0.5985747000920059,4.853170507119505);
\draw [line width=0.8pt,color=ccqqqq] (-0.9483588491124162,4.3775194966905016)-- (-1.269690807371017,4.389016375516796);
\begin{scriptsize}
\draw [fill=ffvvqq] (-0.4999606679691679,1.1470404335189528) circle (1.5pt);
\draw [fill=zzttff] (-0.49068208754217607,2.807906329950484) circle (1.5pt);
\draw [fill=ffvvqq] (3.1842244329397147,1.627837927981798) circle (1.5pt);
\draw [fill=zzttff] (-0.4839225082078124,4.01787103080157) circle (1.5pt);
\draw (0.4390359334533774,3.9874852485355086) node {$A_3$};
\draw (-0.2678837292759286,4.748144473447588) node {$A_2$};
\draw (-1.3858247508494508,4.548144473447588) node {$A_1$};
\draw [fill=ccqqqq] (4.018020006351946,2.893314897557809) circle (2pt);
\draw (3.949975988607363,3.203501509273557) node {$Q_3$};
\draw [fill=ccqqqq] (3.0187093101124347,4.25292128699932) circle (2pt);
\draw (3.3445721034063343,4.279676511713913) node {$Q_1$};
\draw [fill=ccqqqq] (3.3216209694163403,3.8407965804633943) circle (2pt);
\draw (3.64659856035672,3.865788404041164) node {$Q_2$};
\end{scriptsize}
\end{tikzpicture}
\end{center}

\subsection{Removing a finite set from the special fiber} \label{modeli1}

Let $Z \subseteq C$ be a Zariski closed subset such that $X$ has good reduction over $C \backslash Z$.

\begin{cons}[The model $\mathscr{C}_1$] \label{modelI} For any $z \in Z$, there exists a strict affinoid neighborhood $V_z$ of~$z$ in $C$ such that $X(\mathscr{M}(V_z)) \neq \emptyset$. Let $\mathscr{C}_1$ denote a \emph{proper regular model} of $C^{\mathrm{al}}$ over $k^{\circ}$ corresponding to a vertex set $S$ of $C$ such that $\bigcup_{z \in Z} \partial{V_z} \subseteq S$. 
We will denote by $\mathscr{C}_{1,s}$ its special fiber and by $\pi_1$ the corresponding specialization morphism. 
\end{cons}
See Figure 3 for a couple of illustrations of $\mathscr{C}_{1,s}$ and $\pi_1$, where the bijection of Theorem~\ref{lordan} associates to $\eta_i \in S$ the irreducible component $I_i$ of $\mathscr{C}_{1,s}$.

\begin{center}
\definecolor{zzttff}{rgb}{0.6,0.2,1}
\definecolor{qqqqff}{rgb}{0,0,1}
\definecolor{ffvvqq}{rgb}{1,0.3333333333333333,0}
\begin{tikzpicture}[line cap=round,line join=round,>=triangle 45,x=1cm,y=1cm, scale=0.6]
\draw [line width=0.4pt] (0.927784096234669,4.555077671694948)-- (0.905483209156678,-2.177996760609003);
\draw [line width=0.4pt] (0.9186144926466258,1.786594819102608)-- (2.785709689851428,-0.5503145057225194);
\draw [line width=0.4pt] (1.9284535960297908,0.5226515446503277)-- (3.238481473169798,0.46451880171520354);
\draw [line width=0.4pt] (0.9225013737928455,2.9601203374145832)-- (-1.507261850772056,3.8153344131303286);
\draw [line width=0.4pt] (0.9129657388745637,0.08112545978783547)-- (-1.0519537846884799,-1.8507215203541716);
\draw [line width=0.4pt] (-0.3847844522123841,3.4202512786027146)-- (-1.4819669582118573,3.0058978512039713);
\draw [shift={(-1.2839982996141128,3.470879812317323)},line width=0.4pt,color=qqqqff]  plot[domain=-2.134249374401019:1.560908047156625,variable=\t]({0.925847643695199*1.0227093951920805*cos(\t r)+-0.377896997426611*0.6266859925815617*sin(\t r)},{0.377896997426611*1.0227093951920805*cos(\t r)+0.925847643695199*0.6266859925815617*sin(\t r)});
\draw [line width=0.4pt] (-0.6075623452445441,3.4986633592339524)-- (-1.1465596851370305,4.0359052273897715);
\draw [line width=0.4pt] (-0.8951453447492822,3.78530972363827)-- (-0.5357213985722202,3.95955044156917);
\draw [shift={(0.7908485439104334,-2.1149545541303563)},line width=0.4pt,color=qqqqff]  plot[domain=1.5781591313997856:4.762433771929928,variable=\t]({-0.07449966818208802*2.7397071622564027*cos(\t r)+0.9972210384066105*0.7974006734735172*sin(\t r)},{-0.9972210384066105*2.7397071622564027*cos(\t r)+-0.07449966818208802*0.7974006734735172*sin(\t r)});
\draw [shift={(-1.1459424669443607,-1.9050082425218828)},line width=0.4pt,color=qqqqff]  plot[domain=1.882744566759224:4.439873151832653,variable=\t]({-0.7387900377168731*0.8001336383852711*cos(\t r)+0.673935664711626*0.2945809417293864*sin(\t r)},{-0.673935664711626*0.8001336383852711*cos(\t r)+-0.7387900377168731*0.2945809417293864*sin(\t r)});
\draw [line width=0.4pt] (4.422382129975567,4.193397938884286)-- (6.578284069220799,2.0030015686111486);
\draw [line width=0.4pt] (6.423059129595142,3.1585650080465832)-- (4.387887698947643,0.33002166375686276);
\draw [line width=0.4pt] (4.45687656100349,1.330360163566642)-- (6.802497870902303,-1.2394749480136529);
\draw [line width=0.4pt] (6.854239517444189,-0.29087809474575876)-- (4.73283200922688,-2.0155996461419297);
\draw (2.3258450408004404,2.3859780966795148) node[anchor=north west] {$\xrightarrow{\pi_1}$};
\draw (0.17722103690939937,3.501662770067288) node {$\eta_1$};
\draw (0.50768877028758166,0.7215292218543007) node {$\eta_2$};
\draw (-0.1299064066937825,-0.37267249325840144) node  {$\eta_3$};
\draw (-0.9445690541782527,-0.9819644011741126) node  {$\eta_4$};
\draw (4.692765512750745,4.440707518270488) node[anchor=north west] {$I_1$};
\draw (5.447922479660237,2.2103601973982347) node[anchor=north west] {$I_2$};
\draw (5.266933840810193,0.5644283642979496) node[anchor=north west] {$I_3$};
\draw (5.412798899803981,-1.0532335867899188) node[anchor=north west] {$I_4$};
\draw [rotate around={0.8184554616886093:(12.39546184616989,0.9280946496584714)},line width=0.4pt] (12.39546184616989,0.9280946496584714) ellipse (1.2480330567423494cm and 0.8199896966589879cm);
\draw [line width=0.4pt] (13.148744441510996,1.5880084303439157)-- (14.70703495902736,3.4681371980890523);
\draw [line width=0.4pt] (14.048168215451849,2.673192799222254)-- (13.524369645472373,3.817561040730296);
\draw [line width=0.4pt] (13.477749880424149,0.5285997175088331)-- (14.868307501784859,-0.45616134234337985);
\draw [line width=0.4pt] (14.307451315009098,-0.05897583601857592)-- (14.164522513244712,-0.6709262169763286);
\draw [line width=0.4pt] (11.389573320597911,1.4052927073124146)-- (9.788222404923665,2.5811382129228178);
\draw [line width=0.4pt] (11.588778517974852,0.29585940189973603)-- (10.352676304574908,-0.3755250709646313);
\draw [line width=0.4pt] (10.728978904342403,3.1993496268265567)-- (10.654791633139563,1.9448307538117011);
\draw [line width=0.4pt] (13.766610755995323,3.2883250909624926)-- (13.061337277227128,3.517850238908914);
\draw [shift={(10.711598522290544,3.2215993230316986)},line width=0.4pt,color=qqqqff]  plot[domain=1.5902595759905616:4.676352420201576,variable=\t]({0.03771432069835844*0.9522161110225346*cos(\t r)+-0.9992885619350705*0.3768214500759947*sin(\t r)},{0.9992885619350705*0.9522161110225346*cos(\t r)+0.03771432069835844*0.3768214500759947*sin(\t r)});
\draw [shift={(14.55696481678566,-0.5068548363155911)},line width=0.4pt,color=qqqqff]  plot[domain=1.031491463195791:4.332507966733717,variable=\t]({0.8189368223496922*1.0826684496284427*cos(\t r)+0.5738836824651742*0.623708391929002*sin(\t r)},{-0.5738836824651742*1.0826684496284427*cos(\t r)+0.8189368223496922*0.623708391929002*sin(\t r)});
\draw [shift={(10.860920869101705,-0.07410333920287389)},line width=0.4pt,color=qqqqff]  plot[domain=-2.289101442939252:2.405270907717504,variable=\t]({0.9431329885558617*0.8426183781520289*cos(\t r)+-0.33241565230549697*0.38504817873110847*sin(\t r)},{0.33241565230549697*0.8426183781520289*cos(\t r)+0.9431329885558617*0.38504817873110847*sin(\t r)});
\draw [line width=0.4pt] (11.217564925237733,0.09423607368451459)-- (10.557081973361718,0.058699112105521124);
\draw [line width=0.4pt] (10.855995023689662,0.07478198230647357)-- (10.668178088735099,0.22826686714910574);
\draw [line width=0.4pt] (16.95097208629959,2.8735687970109414)-- (19.982924454830396,3.3283616522905697);
\draw [line width=0.4pt] (17.79318107755815,3.3957383715912552)-- (17.102569704726132,0.29640928375971615);
\draw [line width=0.4pt] (16.715153568747194,2.2840225031299424)-- (19.73026175745283,1.0207090162420867);
\draw [line width=0.4pt] (16.301425521820278,0.8448907692880124)-- (21.91053740360227,1.6197230412458972);
\draw [line width=0.4pt] (21.364147200494433,2.1661132443537423)-- (20.75775672678827,-0.4615788083729972);
\draw (14.878603671064822,1.8240008189794195) node[anchor=north west] {$\xrightarrow{\pi_1}$};
\draw (4.338787780550815,-3.1365353719023323) node[anchor=north west] {Figure 3. ${\color{orange} Z, \pi_1(Z)}, {\color{blue} V_z}, z \in Z, \eta_i \in {\color{violet}S}$};
\begin{scriptsize}
\draw [fill=ffvvqq] (-1.507261850772056,3.8153344131303286) circle (2.5pt);
\draw [fill=ffvvqq] (-1.0254218052634183,-1.8246361139470166) circle (2.5pt);
\draw [fill=ffvvqq] (0.905483209156678,-2.1779967606090027) circle (2.5pt);
\draw [fill=zzttff] (0.23215918534341715,-0.5882220980781063) circle (2.5pt);
\draw [fill=zzttff] (-0.5566744587860053,-1.363778488541619) circle (2.5pt);
\draw [fill=zzttff] (0.9147156709489045,0.6094642145111961) circle (2.5pt);
\draw [fill=zzttff] (-0.38987893245843724,3.4220444044393528) circle (2.5pt);
\draw [fill=ffvvqq] (6.112609250343828,-0.8945306377344187) circle (2.5pt);
\draw [fill=ffvvqq] (4.8207813061685565,0.9316704059960945) circle (2.5pt);
\draw [fill=ffvvqq] (5.3977454624438215,3.202428793096548) circle (2.5pt);
\draw [fill=ffvvqq] (14.164522513244712,-0.6709262169763286) circle (2.5pt);
\draw [fill=ffvvqq] (10.728978904342403,3.1993496268265567) circle (2.5pt);
\draw [fill=zzttff] (10.674028685641886,2.270132453605452) circle (2.5pt);
\draw (11.044762727164214,2.2718264621466826) node {$\eta_1$};
\draw [fill=zzttff] (13.86062223201979,0.25745828693906836) circle (2.5pt);
\draw (14.153199544442819,0.3912653686151653) node {$\eta_5$};
\draw [fill=zzttff] (11.588778517974852,0.29585940189973603) circle (2.5pt);
\draw (11.910661173211238,0.4210182089026535) node {$\eta_3$};
\draw [fill=ffvvqq] (10.557081973361718,0.058699112105521124) circle (2.5pt);
\draw [fill=zzttff] (11.389573320597911,1.4052927073124146) circle (2.5pt);
\draw (11.401369265295534,1.6732760333685073) node {$\eta_2$};
\draw [fill=zzttff] (13.477749880424149,0.5285997175088331) circle (2.5pt);
\draw (13.09267183796456,0.5546922175370849) node {$\eta_4$};
\draw[color=black] (18.619264925756028,3.3296687988405695) node {$I_1$};
\draw[color=black] (17.35481605093083,2.574298059271802) node {$I_2$};
\draw[color=black] (18.508523446618492,1.796208562865403) node {$I_3$};
\draw[color=black] (17.86410795884653,0.9249393772495967) node {$I_4$};
\draw[color=black] (20.779365086915735,0.9276247470339806) node {$I_5$};
\draw [fill=ffvvqq] (19.31912509414813,3.2287917481882267) circle (2.5pt);
\draw [fill=ffvvqq] (18.052113805931334,1.7238436328013833) circle (2.5pt);
\draw [fill=ffvvqq] (20.895560523828152,0.13557097879983127) circle (2.5pt);
\end{scriptsize}
\end{tikzpicture}

\end{center}

\begin{rem} \label{2.4}
If  $\bigcup_{z \in Z} V_z =C,$ then for any $x \in C,$ $X(\mathscr{M}_x) \neq \emptyset$ (see Lemma \ref{1.3bis}), and, equivalently, ${X(F_v) \neq \emptyset}$ for all ${v \in V(F)}.$ Hence, without loss of generality, we may assume that $\bigcup_{z \in Z} V_z \neq C.$ 
\end{rem}

\begin{thm} \label{2.5}
If there exists a proper model $\mathcal{X} \rightarrow \mathscr{C}_1$ of $X/F$ which is smooth over $\mathscr{C}_{1,s} \backslash \pi_1(Z)$, then  $X(\mathscr{M}_x) \neq \emptyset$ for all $x \in C$, and, equivalently, $X(F_v) \neq \emptyset$ for all ${v \in V(F)}.$ 
Furthermore, if there exists a rational linear algebraic group $G/F$ acting strongly transitively on $X$, then $X(F) \neq \emptyset$. 
\end{thm}

\begin{rem} \label{2.6}
Let $P$ be any closed point of $\mathscr{C}_1.$ As $\mathcal{O}_{\mathscr{C}_1, P}$ and $\mathcal{O}_{\mathscr{C}_{1,s}, P}$ are local rings with the same residue field, the respective base change of $\mathcal{X}$ is smooth over $\mathcal{O}_{\mathscr{C}_1, P}$ if and only if it is smooth over  $\mathcal{O}_{\mathscr{C}_{1,s}, P}$.
\end{rem}

\begin{proof}[Proof of Theorem \ref{2.5}]
 Let  ${x \in C \backslash (\bigcup_{z \in Z} V_z \cup S)}$. Set ${P_x:=\pi_1(x) \in \mathscr{C}_{1,s}}$. Assume there exists $z_0 \in Z$ such that $P_x=\pi_1(z_0).$ Set $U_x:=\pi_1^{-1}(P_x).$ Then $z_0 \in U_x$. By Theorem~\ref{lordan}, $U_x$ is a connected component of~$C \backslash S.$  By its connectedness, there exists an injective path $[x, z_0]$ connecting $x$ and $z_0$ entirely contained in $U_x$. As $x \not \in V_{z_0}$, the path $[x,z_0]$ must intersect $\partial{V_{z_0}}$, meaning $[x, z_0] \cap S \neq \emptyset.$ As a consequence, $U_x \cap S \neq \emptyset,$ contradiction. Thus, $P_x \not \in \pi_1(Z).$ By Lemma \ref{2.1},  $X(\mathscr{M}_x) \neq \emptyset$. If $x \in \bigcup_{z \in Z} V_z \cup S$, then $X(\mathscr{M}_x) \neq \emptyset$ by assumption. We have shown that $X(\mathscr{M}_x) \neq \emptyset$ for all $x \in C,$ so we can conclude by relation \eqref{artin} and \cite[Cor.~3.18]{une}.
\end{proof}

\begin{rem} \label{2.7}
In light of Remark \ref{2.3}, we can take $Z$ to be any Zariski closed subset of~$C^{\mathrm{al}}$, and then the hypothesis of Theorem \ref{2.5} will imply that $\pi_1^{-1}(\pi_1(Z))$ contains the bad reduction points of $X$ in $C$. We thus assumed directly that $Z$ itself contains them.
\end{rem}

\subsection{Removing irreducible components} \label{modeli2} Let $Z$ be as in Subsection \ref{modeli1}.
We show that by further refining the model $\mathscr{C}_1$ from Section \ref{modeli1}, more points from the special fiber can be forgotten in Theorem \ref{2.5}. We will do this in two steps, the first of which consists of being more restrictive when constructing the neighborhoods $V_z$ of $z \in Z.$

\begin{cond}\label{2.8}
From now on, throughout this section, we will assume that the curve~$C$ is smooth. (This is stronger than normal when $k$ is not a perfect field.)
\end{cond}

In practice, we only need $C$ to be smooth at the points of the subset $Z$. This ensures the existence of special neighborhoods of these points in $C$.

A Berkovich curve has the structure of a real graph (\emph{cf.} \cite[1.3.1]{duc}), so it makes sense to speak of branches issued from a point of $C$. See \cite[1.7]{duc} for some elements from the branch language which we use here. See also \cite[pg. 210]{duc} for the notion of a \emph{virtual disc}.   

\begin{lm} \label{2.10}
Let $N$ be a closed virtual disc in $C$. Then 
\begin{enumerate}
\item  $\partial{N}$ is a single type 2 point $\{\eta\}$,
\item there exists a unique branch $b$ issued from $\eta$ in $C$ not contained in $N$.
\end{enumerate}
If $z \in C$ is a rigid point, then there exists a neighborhood $N$ of $z$ which is a  closed virtual disc and $X(\mathscr{M}(N)) \neq \emptyset$.
\end{lm}

\begin{proof}
From \cite[Th\'eor\`eme 4.5.4]{duc}, virtual discs form a basis of neighborhoods of $z$ in~$C$. As type 2 points are dense, there exists a neighborhood $N$--a closed virtual disc--of $z$ satisfying condition $(1)$ and such that $X(\mathscr{M}(N)) \neq \emptyset.$ To show condition~$(2)$, let~$\widehat{\overline{k}}$ be the completion of an algebraic closure of $k$. Then  the base change $N_{\widehat{\overline{k}}}$ is a finite disjoint union of closed discs embedded in $C_{\widehat{\overline{k}}}$. Let $D$ be one of those discs, and~$\omega$ its unique boundary point. Consider an embedding of $D$ in $\mathbb{P}_{\widehat{\overline{k}}}^{1, \mathrm{an}}$. It suffices to show that there exists a unique branch issued from $\omega$ in  $C_{\widehat{\overline{k}}}$ that is not contained in~$D$ (\emph{cf.} \cite[1.7.2]{duc}). Moreover, one can reduce to the case of $\mathbb{P}_{\widehat{\overline{k}}}^{1, \mathrm{an}}$ (\emph{cf.} \cite[4.2.11.1]{duc}), thus condition $(2)$ is also satisfied. 
\end{proof}

\begin{cons}[The model $\mathscr{C}_2$] \label{model2} For any point $z \in Z$, let $N_z$ denote an affinoid neighborhood of $z$ in $C$ satisfying the conditions of Lemma \ref{2.10}. As in Remark \ref{2.4}, without loss of generality, we may assume  $\bigcup_{z \in Z} N_z\neq C$. We may also assume that $N_z, z \in Z$, are mutually disjoint. Let $s_0 \in C$ be a type 2 point such that $s_0 \not \in \bigcup_{z \in Z} N_z$. Set $T:=\bigcup_{{z \in Z}} \partial{N_z}.$ 

Let $\mathscr{C}_2$ be a proper sncd model of the curve $C^{\mathrm{al}}$ over $k^{\circ}$ such that the corresponding vertex set $S$ of $C$ satisfies $T \cup \{s_0\} \subseteq S$ (\emph{cf.} Corollary \ref{lordan1}).  We denote by $\mathscr{C}_{2,s}$ the special fiber of $\mathscr{C}_2$, and by $\pi_2:C \rightarrow \mathscr{C}_{2,s}$ the specialization morphism. 
\end{cons}
Given a point~$y \in S$, let $I_y$ be the unique corresponding irreducible component of~$\mathscr{C}_{2,s}$. See Figure 4 below for an illustration. We show a statement analogous to that of Section~\ref{modeli1}, starting with an auxiliary result. 

\begin{lm} \label{shtesa}
For any two points $\eta_1, \eta_2 \in T$, one has $I_{\eta_1} \cap I_{\eta_2}=\emptyset$.
\end{lm}

\begin{proof} Let $z_i \in Z$ be such that $\{\eta_i\}=\partial{N_{z_i}}, i=1,2.$
Assume, by contradiction, that there exists $P \in I_{\eta_1} \cap I_{\eta_2}$. As $\mathscr{C}_2$ is an sncd model, $P$ is an ordinary double point, so it is not contained in a third irreducible component of $\mathscr{C}_{2,s}$. Then $U:=\pi_2^{-1}(P)$ is a connected component of $C \backslash S$ and $\partial{U}=\{\eta_1, \eta_2\}.$ Thus, there exists an \emph{open} injective path $\gamma:=(\eta_1, \eta_2)$ contained in $U$ connecting $\eta_1$ and $\eta_2$. The only branch issued from~$\eta_i$ not contained in $N_{z_i}$ is thus contained in $\gamma$, $i=1,2.$   As $s_0 \in S \backslash \bigcup_{z \in Z} \partial{N_z}$, we have $I_{\eta_1} \cup I_{\eta_2} \neq \mathscr{C}_{s,2}.$
So, since $\mathscr{C}_{s,2}$ is connected, there exists a third irreducible component $I_s, s\in S$, intersecting at least one of $I_{\eta_1}$ or $I_{\eta_2}$, and $I_s \cap I_{\eta_1} \cap I_{\eta_2}=\emptyset$. Without loss of generality, let $I_s \cap I_{\eta_1}=:\{Q\}.$  Similarly, $\pi_2^{-1}(Q)=:V$ is a connected component of $C \backslash S$ and $\partial{V}=\{s, \eta_1\}$. Remark that $P \neq Q$ and thus $U \cap V=\emptyset.$ As before, there exists an injective open path $\gamma':=(s, \eta_1)$ contained in $V$, so the unique branch issued from $\eta_1$ outside of $N_{z_1}$ is contained in $\gamma'$. But then $\gamma \cap \gamma' \neq \emptyset$, implying $U \cap V \neq \emptyset$, contradiction.  
\end{proof}

\begin{center}
\definecolor{zzttff}{rgb}{0.6,0.2,1}
\definecolor{ffvvqq}{rgb}{1,0.3333333333333333,0}
\definecolor{qqqqff}{rgb}{0,0,1}
\begin{tikzpicture}[line cap=round,line join=round,>=triangle 45,x=1cm,y=1cm, scale=0.8]
\draw [line width=0.4pt] (-0.48,4.72)-- (-0.5,1.14);
\draw [line width=0.4pt] (-0.48670432557268584,3.519925722489233)-- (1.02,4.66);
\draw [line width=0.4pt] (0.3916767533585498,4.18456818547614)-- (1.3,3.86);
\draw [line width=0.4pt] (-0.4939666687472692,2.219966294238812)-- (-1.24,1.06);
\draw [shift={(-0.48801748977686904,1.172463631037081)},line width=0.4pt,color=qqqqff]  plot[domain=1.5396519666041433:4.782244576907428,variable=\t]({0.0029446125003494243*1.635448426077079*cos(\t r)+0.9999956646192136*0.8724581894977228*sin(\t r)},{-0.9999956646192136*1.635448426077079*cos(\t r)+0.0029446125003494243*0.8724581894977228*sin(\t r)});
\draw [line width=0.4pt] (-0.4952515111644949,1.9899795015554131)-- (0.16716501042297213,1.3018540097524218);
\draw [line width=0.4pt] (-0.32490499689771574,1.8130216886285284)-- (0.1958362125084269,1.8007329260393339);
\draw [line width=0.4pt] (-0.48392250820781246,4.01787103080157)-- (-1.245799103612277,4.607850161634012);
\draw [line width=0.4pt] (2.82398304652662,4.517854978952808)-- (4.293904363652954,2.517962030481606);
\draw [line width=0.4pt] (4.343901687364734,3.387915463066579)-- (2.933977158692536,1.2480300082023923);
\draw (1.5140531652779827,3.0779320560535424) node[anchor=north west] {$\xrightarrow{\pi_2}$};
\draw (-1.384081533933114,0.8780498127352199) node[anchor=north west] {Figure 4. ${\color{orange}Z, \pi_2(Z)}, {\color{blue}{N_z}}, z \in Z, {\color{violet}{S}}$};
\begin{scriptsize}
\draw [fill=ffvvqq] (-0.4999606679691679,1.1470404335189528) circle (2.5pt);
\draw [fill=zzttff] (-0.49068208754217607,2.807906329950484) circle (2.5pt);
\draw (-0.29584406525437183,2.9229349999705843) node {$\eta$};
\draw [fill=zzttff] (-0.4867043255726858,3.5199257224892326) circle (2.5pt);
\draw (-0.23584406525437183,3.5128969966778613) node {$s_0$};
\draw[color=black] (3.463948790037405,3.942907166573097) node {$I_{s_0}$};
\draw[color=black] (3.5439445079762533,2.5229617628527836) node {$I_{\eta}$};
\draw [fill=ffvvqq] (3.1842244329397116,1.6278379279817945) circle (2.5pt);
\end{scriptsize}
\end{tikzpicture}
\end{center}

\begin{thm} \label{2.13}
If there exists a proper model $\mathcal{X} \rightarrow \mathscr{C}_2$ of $X/F$ which is smooth over $\bigcup_{s \in S \backslash T} I_s$, then $X(\mathscr{M}_x) \neq \emptyset$ for all $x \in C$, and, equivalently, $X(F_v) \neq \emptyset$ for all ${v \in V(F)}.$  
Furthermore, if there exists a rational linear algebraic group $G/F$ acting strongly transitively on $X$, then $X(F) \neq \emptyset$. 
\end{thm}

\begin{proof} Recall Remark \ref{2.6}. 
 Let $x \in C$. If $x \in \bigcup_{z \in Z}N_z \cup S$, then $X(\mathscr{M}_x) \neq \emptyset$. Suppose $x \not \in \bigcup_{z \in Z}N_z \cup S.$ Set $U:=C \backslash (\bigcup_{z \in Z} N_z \cup S).$  Let $U_x$ be the connected component of $C \backslash S$ containing $x$. By Theorem \ref{lordan}, $\pi_2(U_x)=:\{P_x\}$, where $P_x$ is a closed point of~$\mathscr{C}_{2,s}$, and ${\pi_2^{-1}(P_x)=U_x}$. Let us show that $P_x \in \bigcup_{s \in S\backslash T} I_s$. 

By \emph{loc.cit.}, $\partial{U_x} \subseteq S$, so depending on whether $P_x$ is a double point or not, $\partial{U_x}$ consists either of two points or one.  
If $P_x$ is a double point, there is nothing to check, as by Lemma \ref{shtesa}, the set $\bigcup_{s \in S \backslash T} I_s$ contains all the double points of $\mathscr{C}_{2,s}$.
Suppose that $\partial{U_x}$ is a singleton. Remark that $P_x \in \bigcup_{i \in T} I_i$ if and only if $\partial{U_x} \in T$ (\emph{cf.} Theorem \ref{lordan}). Assume, by contradiction, this is the case. Let $z_0 \in Z$ be such that $\partial{U_x}=\partial{N_{z_0}}=:\{\eta\}.$  There are three possibilities:

(1) $N_{z_0} \cup U_x=C,$ in which case $S \backslash \{\eta\} \subseteq N_{z_0},$ meaning $s_0 \in N_{z_0},$ contradiction; 

(2) $U_x \subseteq N_{z_0},$ in which case $x \in N_{z_0}$, contradiction;

(3) $N_{z_0} \cup U_x\neq C$ and $U_x \not \subseteq N_{z_0};$ let $a \in C \backslash (N_{z_0} \cup U_x);$ let $[a, \eta)$ be an injective path in~$C$ connecting $a$ and $\eta$ (but without containing~$\eta$); it is entirely contained in $C \backslash (N_{z_0} \cup U_x)$.
 As $a \not \in N_{z_0}$, the path $[a,\eta)$ contains the unique branch~$b$ issued from~$\eta$ that is outside of~$N_{z_0}.$ Let us show that $b \subseteq U_x.$  Let $c \in U_x \backslash N_{z_0}$. There exists an injective path $[c, \eta)$ in $U_x$ connecting $c$ and~$\eta$, without containing $\eta$. This path intersects a branch issued from $\eta$ and is disjoint from $N_{z_0}$, so $b \subseteq [c, \eta)$, meaning $[c, \eta) \cap [a, \eta) \neq \emptyset,$ contradicting $[a, \eta) \cap U_x=\emptyset.$ 

To resume, we have shown that $\pi_2(x)=\pi_2(U_x)=P_x \in \bigcup_{s \in S \backslash T} I_s$. This means that for any $x \in U$, $X$ has a proper smooth model over $\mathcal{O}_{\mathscr{C}_2, P_x},$ hence by Lemma \ref{2.1}, $X(\mathscr{M}_x) \neq \emptyset$. We can now conclude by relation \eqref{artin} and \cite[Cor.~3.18]{une}.
\end{proof}

\begin{rem}\label{2.14}
Let us briefly explain why, in order to apply the techniques of this manuscript, the point $s_0$ in Construction \ref{model2} is necessary. Let us look at the example illustrated in Figure 4. If $S:=\{\eta\}=T$ instead of $\{s_0, \eta\}$, then $\mathscr{C}_{1,s}$ contains a unique irreducible component $I_{\eta}$, making the hypothesis of Theorem \ref{2.13} empty. If we took $s_0 \in N_{z}$, then the hypothesis ensures good reduction of $X$ over~$I_{s_0}$, but the points in $C \backslash N_z$ all map to $I_{\eta} \backslash I_{s_0}$ via $\pi_2$, meaning we can't apply this technique to them and they remain unaccounted~for. 
\end{rem}

\begin{rem} \label{2.17}
If $T=\emptyset,$ meaning $Z=\emptyset$, then the variety $X/F$ has good reduction over all points of the curve $C^{\mathrm{al}}$. In that case, in Theorems \ref{2.5} and~\ref{2.13}, we have to check that~$X$ has a proper smooth model over the entire special fiber $\mathscr{C}_{2,s}$ of the model~$\mathscr{C}_2$ of~$C^{\mathrm{al}}$. Hence, this condition is directly related to the uniformizer~$t$ of~$k^{\circ}.$ More precisely, we check smoothness of $X$ via the nonvanishing of certain minors~$\epsilon$ of a matrix defined over ~$F$. As~$\epsilon$ doesn't vanish anywhere on $C^{\mathrm{al}}$, it is constant, meaning defined over $k$. Then, checking whether $X$ has a model that is smooth over $\mathscr{C}_{2,s}$ comes down to checking whether $\epsilon$ is invertible in $k^{\circ}$, or equivalently, whether it is non-zero on the residue field $\widetilde{k}$ of $k$.
\end{rem}

\section{A smoothness criterion over residue fields of completions} \label{seksioni3}

We denote by $\mathcal{H}(\cdot)^{\circ}$ the valuation ring, and by $\widetilde{\mathcal{H}(\cdot)}$ the residue field of the completed residue field~$\mathcal{H}(\cdot)$. 

\subsection{Over the analytic curve} \label{ana} 
We remark that for an open $U$ of $C$ and a point $x \in U$, there is a map $\mathcal{O}^{\circ}(U) \rightarrow \widetilde{\mathcal{H}(x)}$ induced by the natural map $\mathcal{O}^{\circ}(U) \rightarrow \mathcal{H}(x)^{\circ}$. As before, we start by proving local statements.

\begin{prop} \label{3.1}
Let $x \in C$. Suppose there exists a strict affinoid neighborhood $U$ of $x$ in $C$ such that $X$ has a proper model $\mathcal{X} \rightarrow \mathrm{Spec} \ \mathcal{O}^{\circ}(U)$ which is smooth over $\widetilde{\mathcal{H}(y)}$ for all but a finite number of rigid and type 2 points $y$ of $U$.
 Then there exists a neighborhood $U_x \subseteq U$ of $x$ such that $X(\mathscr{M}(U_x)) \neq \emptyset$, and hence $X(\mathscr{M}_x) \neq \emptyset.$
\end{prop}

\begin{proof} As already remarked, it makes sense to consider the model $\mathcal{X}$ of $X$ (by applying a base change) over $\widetilde{\mathcal{H}(y)}, y\in U.$
By restricting to a smaller neighborhood of $x$ if necessary, as $C$ is separated, we may assume that $\mathcal{X}$ is smooth over $\widetilde{\mathcal{H}(y)}$ for \emph{all} rigid and type 2 points $y$ of $U$. Let $\mathscr{C}$ be a proper regular model of $C^{\mathrm{al}}$ over $k^{\circ}$ corresponding to a vertex set $S$ of $C$ such that $\partial{U} \subseteq S$ (see Corollary \ref{lordan1}). Let $\mathscr{C}_s$ denote its special fiber and $\pi$ the specialization morphism $C \rightarrow \mathscr{C}_s.$

Set $P_x:=\pi(x)$ and $U_x:=\pi^{-1}(P_x).$ As $\partial{U_x} \subseteq S$, so $\partial{U_x}$ contains only type 2 points. 
By the proof of Lemma \ref{ekstra}, $U_x \subseteq U$, and so for any rigid or type 2 point $\eta \in U_x$, the model $\mathcal{X}$ of $X$  is smooth over $\widetilde{\mathcal{H}(\eta)}.$ From \cite[2.4, pg. 35]{Ber90} (see also Remark~\ref{berred}), the specialization map $\pi$ induces an embedding $\kappa(P_x) \hookrightarrow \widetilde{\mathcal{H}(\eta)}$, so $\mathcal{X}$ is smooth over $\kappa(P_x)$, and thus over $\widehat{\mathcal{O}_{\mathscr{C}, P_x}}=\mathcal{O}^{\circ}(U_x).$ Hence, $X(\mathscr{M}(U_x)) \neq \emptyset$ (\emph{cf.} Lemma \ref{1.5.2}), and  $X(\mathscr{M}_x) \neq \emptyset.$
\end{proof}

\begin{rem} \label{3.2}
As we can see from the proof, the hypotheses of Proposition \ref{3.1} can be relaxed to: \emph{for any neighborhood $M$ of $x$ in $C$ there exists a rigid or type 2 point $y_M \in M$ such that $\mathcal{X}$ is smooth over $\widetilde{\mathcal{H}(y_M)}$}. 
\end{rem}

Let us now give a similar global version of Proposition \ref{3.1} akin to Corollary \ref{1.7}. 

\begin{cor} \label{3.3}
Let $Q$ be a finite set of rigid and type 2 points of $C$. For any $z \in Q$, let $V_z$ be a strict affinoid neighborhood of $z$ in $C$ such that $X(\mathscr{M}(V_z)) \neq \emptyset.$ Set $U:=C \backslash (\bigcup_{z \in Q} V_z  \cup A),$ where $A$ is a finite set of type 2 points. If $X$ has a proper model $\mathcal{X} \rightarrow \mathrm{Spec} \ \mathcal{O}^{\circ}(U)$ such that  $\mathcal{X}$ is smooth over $\widetilde{\mathcal{H}(y)}$ for all but a finite number of \emph{rigid} and \emph{type 2} points $y \in U,$
then $X(\mathscr{M}_x) \neq \emptyset$ for all $x \in C$. Equivalently, $X(F_v) \neq \emptyset$ for all $v \in V(F).$

If, moreover, there exists a rational linear algebraic group $G/F$ acting strongly transitively on $X$, then $X(F) \neq \emptyset$. 
\end{cor}

\begin{proof}
By Proposition \ref{3.1}, for all ${x \in U}$, ${X(\mathscr{M}_x) \neq \emptyset}$. By construction, for any $x \in \bigcup_{z \in Q} V_z \cup A$, ${X(\mathscr{M}_x) \neq \emptyset}$. Thus, $X(\mathscr{M}_x) \neq \emptyset$ for all $x \in C,$ so we can conclude by relation \eqref{artin} and \cite[Cor.~3.18]{une}.
\end{proof}

It suffices to show that for a \emph{finite} number of open subsets $U \subseteq C$, there exists $x_U \in U$ a rigid or type 2 point, such that~$X$ has a  proper model over $\mathcal{O}^{\circ}(U)$ smooth over~$\widetilde{\mathcal{H}(x_U)}$.

\begin{rem} \label{3.3.2}
With the notation of Corollary \ref{3.3}, let $S$ be any vertex set of~$C$ corresponding to a proper regular model $\mathscr{C}$ of $C^{\mathrm{al}}$ and such that $\bigcup_{z \in Q} \partial{V_z} \subseteq S$. By assumption, for any $s \in S$, $X(\mathscr{M}_s) \neq \emptyset$. Hence, there exists an open neighborhood~$U_s$ of $s$ in $C$ such that $X(\mathscr{M}(U_s)) \neq \emptyset.$ By Remark \ref{1.11}, there exist only finitely many connected components $A_1, A_2, \dots, A_n$ of $C \backslash S$ which aren't entirely contained in~${\bigcup_{s \in S} U_s}$; see also~Fig.~2. 

\begin{prop} \label{3.3.3}
If for any $i \in \{1, 2, \dots, n\}$ there exists a rigid or type 2 point $x_i \in A_i$ such that $X$ has a proper model $\mathcal{X} \rightarrow \mathrm{Spec} \ \mathcal{O}^{\circ}(A_i)$ which is smooth over $\widetilde{\mathcal{H}(x_i)}$, then $X(\mathscr{M}_x) \neq \emptyset$ for all $x \in C$. Equivalently, $X(F_v) \neq \emptyset$ for all $v \in V(F)$.
If, moreover, there exists a rational linear algebraic group $G/F$ acting strongly transitively on~$X$, then~${X(F) \neq \emptyset}$. 
\end{prop}

\begin{proof}
If $x \in \bigcup_{s \in S} U_s \cup \bigcup_{z \in Q} V_z$, then $X(\mathscr{M}_x) \neq \emptyset$. Otherwise, suppose there exists $i_0 \in \{1,2,\dots, n\}$ such that $x \in A_{i_0}$. By Theorem \ref{lordan}, $\pi(A_{i_0})=\pi(x)=\pi(x_{i_0})=:P_{i_0} \in \mathscr{C}_s$, where $\mathscr{C}_s$ is the special fiber of $\mathscr{C}$ and $\pi$ the corresponding specialization morphism.
From \cite[2.4, pg. 35]{Ber90} (see also Remark \ref{berred}), $\pi$ induces an embedding $\kappa(P_{i_0}) \subseteq \widetilde{\mathcal{H}(x_{i_0})}$, where $\kappa(P_{i_0})$ is the residue field of $P_{i_0}$ and of the local ring $\mathcal{O}^{\circ}(A_{i_0})=\widehat{\mathcal{O}_{\mathscr{C}, P_{i_0}}}$. Consequently,~$\mathcal{X}$ (or rather, its respective base change) is smooth over~$\kappa(P_{i_0})$, meaning $\mathcal{X}$ is smooth over~$\mathcal{O}^{\circ}(A_{i_0})$. By Lemma \ref{1.5.2}, this implies that ${\mathcal{X}(\mathcal{O}^{\circ}(A_{i_0})) \neq \emptyset}$, hence that $X(\mathscr{M}_x) \neq \emptyset$. We can now conclude by relation \eqref{artin} and \cite[Cor.~3.18]{une}.
\end{proof}

\end{rem}

\subsection{Over a model of the algebraic curve}
We translate the results of Section \ref{ana} over models of $C^{\mathrm{al}}$ and using valuations.

\begin{rem} \label{3.4} 
(1) Let $v$ be a discrete valuation on $F$ such that $v_{|k}$ is trivial. The residue field $\kappa(v)=\widetilde{F_v}$ is a finite field extension of $k$.  In particular, it is uniquely endowed with a discrete valuation extending that of $k$. As a consequence, it makes sense to look at its residue field, which we will call the \emph{double residue field} of $F_v$ and denote by $\widetilde{\widetilde{F_v}}.$

(2) 
Let $v$ be a discrete valuation on $F$. Let $x_v \in C$ be the unique point corresponding to the valuation $v$. By the proof of \cite[Prop.~3.15]{une}, if $v$ extends the norm on $k$, then $F_v =\mathcal{H}(x_v)$, so $\widetilde{F_v}=\widetilde{\mathcal{H}(x_v)}$. If, on the other hand, $v$ is trivial on $k$, then by \emph{loc.cit.}, $\widetilde{F_v}=\mathcal{H}(x_v)$, so $\widetilde{\widetilde{F_v}}=\widetilde{\mathcal{H}(x_v)}.$
\end{rem}
For the next statement, we work with the models $\mathscr{C}_1$ and $\mathscr{C}_2$ of $C^{\mathrm{al}}$ from  Subsections~\ref{modeli1} and \ref{modeli2}, respectively. 

\begin{thm} \label{3.6}
Let $\mathcal{X} \rightarrow \mathscr{C}_1$ be a proper model of $X$. Assume that for all \emph{but a finite number} of discrete $v \in V(F)$:
\begin{enumerate}
\item if $v_{|k}$ is trivial with center $c_v$ satisfying $\pi_1(c_v) \in \mathscr{C}_1 \backslash \pi_1(Z)$, then $\mathcal{X} \rightarrow \widetilde{\widetilde{F_v}}$ is smooth, 
\item if $v_{|k}$ induces the norm on $k$ and with center $c_v$ satisfying $c_v \in \mathscr{C}_{1,s} \backslash \pi_1(Z),$ then $\mathcal{X} \rightarrow \widetilde{F_v}$ is smooth.
\end{enumerate} 
Then, $X(F_v) \neq \emptyset$ for all $v \in V(F)$.
Furthermore, if there exists a rational linear algebraic group $G/F$ acting strongly transitively on $X$, then $X(F) \neq \emptyset$. 
\end{thm}

\begin{rem} \label{3.7}
If $v$ is a discrete valuation on $F$ which is trivial on $k$, \emph{i.e.} induced by a rigid point $x_v$ of $C$, then $c_v=x_v$ (\emph{cf.} \cite[Remark 8.3.19]{liulibri}). 
If $v$ is a discrete valuation on~$F$ which extends the norm of $k$, then by \emph{loc.cit.}, $c_v=\pi_1(x_v)$ (see also Lemma \ref{2.2}). 
\end{rem}

\begin{proof}[Proof of Theorem \ref{3.6}] For any point $P \in \mathscr{C}_{1,s}$,
the morphism $\mathcal{X} \rightarrow \mathscr{C}_1$ induces, via base change, a morphism $\mathcal{X} \rightarrow \mathrm{Spec} \ \mathcal{O}_{\mathscr{C}_1, P}$, hence one $\mathcal{X} \rightarrow \mathrm{Spec} \ \kappa(P)$. For any $x \in C$, from \cite[2.4, pg. 35]{Ber90} (see also Remark \ref{berred}), the specialization morphism~$\pi_1$ induces an embedding $\kappa(P_x) \subseteq \widetilde{\mathcal{H}(x)}$, where $P_x:=\pi_1(x).$ As a consequence, $\mathcal{X}$ gives rise to a model over $\widetilde{\mathcal{H}(x)}$. By Remark \ref{3.4}, this implies the existence of such a model over all $\widetilde{\widetilde{F_v}}$ (resp. $\widetilde{F_v}$) in the statement. By \eqref{artin}, it suffices to show that $X(\mathscr{M}_x) \neq \emptyset$ for all $x \in C$. This is true for any $x \in \bigcup_{z \in Z} V_z \cup S$ by construction. Let ${x \in C \backslash (\bigcup_{z \in Z} V_z \cup S)}$. Set $P_x:=\pi_1(x) \in \mathscr{C}_{1,s}$, and $U_x:=\pi_1^{-1}(P_x).$ By Theorem~\ref{lordan}, $U_x$ is a connected open subset of $C$ and $\partial{U_x} \subseteq S$. 
By the proof of Theorem \ref{2.5}, $P_x \not \in \pi_1(Z),$ so for all but a finite number of rigid and type 2 points $y \in U_x$, the model $\mathcal{X}$ is smooth over~$\widetilde{\mathcal{H}(y)}$. Hence, by Proposition~\ref{3.1}, $X(\mathscr{M}_x) \neq \emptyset.$ We can conclude by relation \eqref{artin} and \cite[Cor.~3.18]{une}. 
\end{proof}

\begin{rem}
The proof of Theorem \ref{3.6} works \emph{mutatis mutandis} if we take the model~$\mathscr{C}_2$ instead of $\mathscr{C}_1$, and replace $\mathscr{C}_{1,s} \backslash \pi_1(Z)$ with $\bigcup_{s \in S\backslash T} I_s$ (see Construction \ref{model2}).
\end{rem}

\section{The case of quadratic forms} \label{quad}
We use the techniques of Section \ref{section1} to show a Hasse principle for quadratic forms in non-dyadic residue characteristic. This is a result originally shown in \cite[Theorem~3.1]{ctps}, where the authors also use local rings of dimension $2$ induced by points on special fibers of models of curves. 

We start by proving a few (standard) results on sncd models of curves. For an sncd model $\mathcal{C}$ of $C^{\mathrm{al}}$, let $\mathcal{I}_{\mathcal{C}_s}$ the invertible ideal sheaf of $\mathcal{O}_{\mathcal{C}}$ defining the special fiber~$\mathcal{C}_s$. For $a \in F^{\times}$, let $[a]$ be the Weil divisor of $a$ in $C^{\mathrm{al}}$, and $\overline{[a]}$ its Zariski closure in $\mathcal{C}.$ Let $\mathrm{div}(a)$ be the Weil divisor of $a$ in $\mathcal{C}$. We remark that $\mathrm{div}(a) \cap C^{\mathrm{al}}=[a]$.

\begin{lm} \label{4.2}
Let $P \in \mathcal{C}_s$ be a closed point. Let $a \in F^{\times}$ such that $P \not \in \overline{[a]}.$
If~${a \in \mathcal{O}_{\mathcal{C},P}},$ then either $a \in \mathcal{I}_{\mathcal{C}_s,P}$ or~$a \in \mathcal{O}_{\mathcal{C}, P}^{\times}.$ 
\end{lm}

\begin{proof}
Suppose $a \not \in \mathcal{I}_{\mathcal{C}_s,P}$; it suffices to show $P \not \in \mathrm{div}(a)$. Suppose $P \in \mathrm{div}(a)$ and let~$I$ be an irreducible component of $\mathrm{div}(a)$ containing $P$. If $I \cap C^{\mathrm{al}} =\emptyset$, then $I \subseteq \mathcal{C}_s$, which is impossible as $a \not \in \mathcal{I}_{\mathcal{C}_s,P}.$ Hence, $I \cap C^{\mathrm{al}} \neq \emptyset$, so there exists a closed point $z \in C^{\mathrm{al}}$ such that $z \in \mathrm{div}(a)$, implying $z \in [a]$. If the closure of~$\{z\}$ in~$\mathcal{C}$ is $\{z, Q\}$ (see \cite[Def.~10.1.31]{liulibri}), then $\{Q\} \subsetneq \{z, Q\} \subseteq I$, so by an argument of dimension: $\{z,Q\}=I$ and $P=Q$. Then $I \subseteq \overline{[a]}$, contradicting the assumption~${P \not \in \overline{[a]}}$. 
\end{proof}

\begin{lm} \label{4.3} Let $P \in \mathcal{C}_s$ be a closed point. There exist generators $\alpha, \beta \in \mathcal{O}_{\mathcal{C}, P}$ of the maximal ideal such that for any $a \in F^{\times}$ for which $a \in \mathcal{O}_{\mathcal{C},P}$ and $P \not \in \overline{[a]}$: 
\begin{enumerate}
\item if $P$ is not a double point of $\mathcal{C}_s$, then  either $\mathcal{I}_{\mathcal{C}_s,P}=\alpha \mathcal{O}_{\mathcal{C}, P}$ or $\mathcal{I}_{\mathcal{C}_s,P}=\beta \mathcal{O}_{\mathcal{C}, P}$, and there exist ${n \in \mathbb{N} \cup \{0\}},$ $u \in \mathcal{O}_{\mathcal{C}, P}^{\times}$ such that $a=u\alpha^n$, resp. $a=u \beta^n$;
\item if $P$ is a double point of $\mathcal{C}_s$, then $\mathcal{I}_{\mathcal{C}_s,P}=\alpha\beta \mathcal{O}_{\mathcal{C}, P}$, and there exist ${m \in \mathbb{N} \cup \{0\}},$ $v \in \mathcal{O}_{\mathcal{C}, P}^{\times}$ such that $a=v\alpha^m \beta^m.$
\end{enumerate} 
\end{lm}

\begin{proof}
For: either $\mathcal{I}_{\mathcal{C}_s,P}=\alpha \mathcal{O}_{\mathcal{C}, P}$ or $\mathcal{I}_{\mathcal{C}_s,P}=\beta \mathcal{O}_{\mathcal{C}, P}$ (resp. $\mathcal{I}_{\mathcal{C}_s,P}=\alpha \beta\mathcal{O}_{\mathcal{C}, P}$) see \cite[Tag~0BI9]{stacks}. For (1), assume, without loss of generality, that $\mathcal{I}_{\mathcal{C}_s,P}=\alpha \mathcal{O}_{\mathcal{C}, P}$.
There exist $n \in \mathbb{N} \cup \{0\}$ and $u \in \mathcal{O}_{\mathcal{C}, P}$ (resp. $m \in \mathbb{N} \cup \{0\}$ and $v \in \mathcal{O}_{\mathcal{C}, P}$ for (2)) such that $a=u\alpha^n$ (resp. $a=v\alpha^m \beta^m$) and $u \not \in \mathcal{I}_{\mathcal{C}_s,P}$ (resp. $v \not \in \mathcal{I}_{\mathcal{C}_s,P}$). As $\mathrm{Frac} \ \mathcal{O}_{\mathcal{C}, P}=F,$ one obtains $u \in F^{\times}$ (resp. $v \in F^{\times}$). Assume $P \in \overline{[u]}$. Then there exists $z \in [u]$ such that the closure of $\{z\}$ in $\mathcal{C}$ is $\{z,P\}$. As $u, \alpha \in \mathcal{O}_{\mathcal{C}, P} \subseteq \mathcal{O}_{C^{\mathrm{al}},z}$, neither $u$ or $\alpha$ have poles on $z$, implying $u$ has a zero on $z$ which is not a pole of $\alpha$. This means that $z  \in [a],$ contradiction because then $P \in \overline{[a]}$. Hence, $P \not \in \overline{[u]}$. (Using the same arguments in (2), $P \not \in \overline{[v]}$). By Lemma~\ref{4.2}, $u \in \mathcal{O}_{\mathcal{C},P}^{\times}$ (resp.~${v \in \mathcal{O}_{\mathcal{C}, P}^{\times}}$).
\end{proof}

\begin{cor} \label{4.4}
Let $P \in \mathcal{C}_s$ be a closed point. There exist generators $\alpha, \beta \in \mathcal{O}_{\mathcal{C}, P}$ of the maximal ideal such that for any $a \in F^{\times}$ satisfying $P \not \in \overline{[a]}$:
\begin{enumerate}
\item if $P$ is not a double point of $\mathcal{C}_s$, then there exist $n \in \mathbb{Z}, u \in \mathcal{O}_{\mathcal{C}, P}^{\times}$ such that either $a=u\alpha^n$ or $a=u \beta^n$;
\item if $P$ is a double point of $\mathcal{C}_s$, then there exist $m \in \mathbb{Z}, v \in \mathcal{O}_{\mathcal{C},P}^{\times}$ such that $a=v \alpha^m \beta^m$.
\end{enumerate}
\end{cor}

\begin{proof}
As $a \in F=\mathrm{Frac} \ \mathcal{O}_{\mathcal{C}, P},$ there exist $b, c \in \mathcal{O}_{\mathcal{C}, P}$ such that $a=b/c$. Moreover, $\mathcal{O}_{\mathcal{C}, P}$ is a regular local ring, so a unique factorization domain, hence without loss of generality, we may assume that $b,c$  have no common prime divisors.  Assume $P \in \overline{[b]} \cup \overline{[c]}.$ Then there exists $z \in [b] \cup [c]$ whose closure in $\mathcal{C}$ is $\{z, P\}.$ Since $P \not \in \overline{[a]},$ we have $z \not \in [a]$. As $[a]=[b]-[c]$, we obtain $z \in [b] \cap [c].$ This implies $b, c \in m_{C^{\mathrm{al}},z}$--the maximal ideal of $\mathcal{O}_{C^{\mathrm{al}},z}.$ Let $m_{\mathcal{C},P}$ denote the maximal ideal of $\mathcal{O}_{\mathcal{C}, P}$. As $z$ is the generic point of $\{z, P\},$ there exists a canonical embedding $\mathcal{O}_{\mathcal{C}, P} \hookrightarrow \mathcal{O}_{C^{\mathrm{al}},z}$.  As $t \not \in m_{C^{\mathrm{al}},z},$  $\mathcal{O}_{\mathcal{C}, P} \cap m_{C^{\mathrm{al}},z}$ is a prime, non-maximal, ideal of the ring $\mathcal{O}_{\mathcal{C}, P}$. As such, it is principal; set  $\mathcal{O}_{\mathcal{C}, P} \cap m_{C^{\mathrm{al}},z} =:q \mathcal{O}_{\mathcal{C}, P}$. Then $b,c \in q \mathcal{O}_{\mathcal{C}, P}$, which contradicting the assumption that $b,c$ are coprime.
Hence, $P \not \in \overline{[b]} \cup \overline{[c]},$ and we can conclude by Lemma~\ref{4.3}.
\end{proof}

We recall that given a commutative ring $R$ and a $q$ a quadratic form defined over it, $q$ is is said to be \emph{$R$-isotropic} if it has a non-trivial zero over $R$. 

\begin{thm} \label{4.5}
Suppose $\mathrm{char} \ \widetilde{k} \neq 2$. Let $q/F$ be a quadratic form which is isotropic over~$F_v$ for all \emph{discrete} $v \in V(F)$. Then $q$ is isotropic over $F_v$ for \emph{all} $v \in V(F)$. Furthermore, if $\dim{q} \neq 2,$ then~$q$ is $F$-isotropic.
\end{thm}

\begin{proof}
By Witt decomposition (\cite[I.4.1]{lam}), we may assume $q$ is a regular quadratic form. Since $\mathrm{char} \ F \neq 2,$ we may  assume $q$ is diagonal with coefficients ${a_1, a_2, \dots, a_n\in F^{\times}}.$ By relation \eqref{artin}, it suffices to show $q$ is $\mathscr{M}_x$--isotropic for all~$x \in C$. Let $Z \subseteq C^{\mathrm{al}}$ be Zariski closed such that $\bigcup_{i=1}^n [a_i] \subseteq Z$; it can be identified to a finite set of rigid points in $C$. For $z \in Z$, let $V_z$ be a strict affinoid 
neighborhood of $z$ in $C$ such that $q$ is $\mathscr{M}(V_z)$--isotropic. Let $\mathscr{C}$ be a proper sncd model of $C^{\mathrm{al}}$ over $k^{\circ}$ corresponding to a vertex set $S$ of 
$C$ such that $\bigcup_{z \in Z} \partial{V_z} \subseteq S$. Let $\mathscr{C}_s$ denote its special fiber and $\pi: C \rightarrow 
\mathscr{C}_s$ the specialization morphism. This is the model $\mathscr{C}_1$ for the quadric $X$ determined by $q$; see Subsection~\ref{modeli1}.

Let $x \in C.$ If $x \in \bigcup_{z \in Z}V_z \cup S,$ then $q$ is $\mathscr{M}_x$-isotropic. Suppose ${x \not \in {\bigcup_{z \in Z} V_z \cup S}}.$ Set $\pi(x)=:P_x$ and ${\pi^{-1}(x)=:U_x}$. By the proof of Theorem \ref{2.5}, $P_x \not \in \pi(Z).$
Let $\alpha, \beta \in \mathcal{O}_{\mathscr{C}, P_x}$ be generators of the maximal ideal.  By Lemma \ref{2.2}, $Z \cup \pi(Z)$ is the Zariski closure of~$Z$ in $\mathscr{C}.$ As $P_x \not \in Z \cup \pi(Z),$ we obtain that $P_x \not \in \bigcup_{i=1}^n \overline{[a_i]}.$ Let $a \in \{a_1, a_2, \dots, a_n\}$. By Corollary \ref{4.4}, there exist $n \in \mathbb{Z}$ and $u_a \in \mathcal{O}_{\mathscr{C}, P_x}^{\times}$ such that $a=u_a \alpha^n$  or $a=u_a\beta^n$ if~$P_x$ is \emph{not} a double point of $\mathscr{C}_s$, resp. $a=u_a \alpha^n \beta^n$ if~$P_x$ \emph{is} a double point of $\mathscr{C}_s.$

By Theorem \ref{bosh}, $\mathcal{O}^{\circ}(U_x)=\widehat{\mathcal{O}_{\mathscr{C}, P_x}}$, so $q$ is $F$-isomorphic to $q':=q_1 \bot \alpha q_2 \bot \beta q_3 \bot \alpha\beta q_4$, where $q_1, q_2, q_3, q_4$ are diagonal quadratic forms over $\mathcal{O}_{\mathscr{C}, P_x}^{\times} \subseteq \mathcal{O}^{\circ}(U_x)^{\times}$. As $\mathrm{Frac} \ \widehat{\mathcal{O}^{\circ}(U_x)_{(\alpha)}}$ is a complete discretely valued field (\cite[Tag 0AFS]{stacks}) containing $F$, the quadratic forms~$q$ and hence $q'$ are $\mathrm{Frac} \ \widehat{\mathcal{O}^{\circ}(U_x)_{(\alpha)}}$--isotropic. 
We may thus assume $q'$ is isotropic over the valuation ring $\widehat{\mathcal{O}^{\circ}(U_x)_{(\alpha)}}$. Note that $\alpha$ is one of its uniformizers. By a theorem of Springer (\cite[VI, Prop. 1.9]{lam}), $q'$  is $\widehat{\mathcal{O}^{\circ}(U_x)_{(\alpha)}}$-isotropic if an only if $q_1 \bot \beta q_3$ or $q_2 \bot \beta q_4$ is isotropic over $\mathcal{O}^{\circ}(U_x)_{(\alpha)}/(\alpha)=\mathrm{Frac} \ \mathcal{O}^{\circ}(U_x)/(\alpha).$ We may assume $q_1 \bot \beta q_3$ is  isotropic over $\mathcal{O}^{\circ}(U_x)/(\alpha).$ By \emph{loc.cit.}, as $\mathcal{O}^{\circ}(U_x)/(\alpha)$ is a discrete valuation ring with uniformizer~$\beta$ (\cite[Tag 00NQ]{stacks}), $q_1$ or $q_3$ is isotropic over $(\mathcal{O}^{\circ}(U_x)/(\alpha))/(\beta)=\mathcal{O}^{\circ}(U_x)/(\alpha, \beta)$. Let us assume that $q_1$ is isotropic over $\mathcal{O}^{\circ}(U_x)/(\alpha, \beta)$. As $q_1$ is defined over  $\mathcal{O}^{\circ}(U_x)^{\times}$, the associated quadric is smooth over $\mathcal{O}^{\circ}(U_x).$ By Hensel's Lemma,~$q_1$ is $\mathcal{O}^{\circ}(U_x)$-isotropic, implying $q'$ is $\mathcal{O}^{\circ}(U_x)$-isotropic.  Consequently, $q$ is $\mathscr{M}(U_x)$-isotropic, and hence $\mathscr{M}_x$-isotropic. 

Thus, $q$ is $\mathscr{M}_x$--isotropic for all $x \in C.$ We can now conclude via \cite[Thm.~3.12]{une}.
\end{proof}

\begin{rem}\label{4.6}
In the proof of Theorem \ref{4.5}, if $P_x$ is not a double point, then~$q'$ is of the form $q_1 \bot \alpha q_2 \bot \beta q_3$ and, in general, both applications of the Theorem of Springer are necessary. However, if $P_x$ \emph{is} a double point, then  the quadratic form $q'$ is of the kind $q_1 \bot \alpha \beta q_2$, so the first application of the Theorem of Springer is enough to conclude. 
\end{rem}

\section{The case of constant varieties} \label{seksioni5}

The techniques and approach presented here differ from those of the previous sections. Let $y \in C$ be such that $X(\mathscr{M}_y)\neq \emptyset$. For a suitably chosen $x \in C,$ we construct an automorphism $\varphi$ of  $C$ such that $x \mapsto y$, with the purpose of showing $X(\mathscr{M}_x) \neq \emptyset$. To ensure such an implication, as $\varphi$ does not fix $F$, we assume $X$ is defined over the smaller~$k$. See Remark \ref{intro} for more details.
We also use different notations. 

\begin{nota} \label{5.0}
(1) Let $k$ be a complete non-trivially valued ultrametric field. Let $\widehat{\overline{k}}$ be the completion of an algebraic closure $\widehat{k}$ of $k$.

(2) For a $k$-analytic space $Y$ and a \emph{valued field extension} $l/k$ (meaning $l$ is complete with a norm extending that of $k$),  $Y_l$ is the $l$-analytic space $Y \times_k l$.
\end{nota}
Recall that to $a \in k$ and $r \in \mathbb{R}_{\geqslant 0},$  one associates a unique point $\eta_{a,r} \in \mathbb{P}_{k}^{1, \mathrm{an}}$, which is a semi-norm on $k[T]$. It is the Shilov boundary of the disc centered at $a$ and of radius $r$. See Def.~2.2 and Prop. 2.3 of \cite{une} for details. The following is a well-known auxiliary result.

\begin{lm} \label{5.2} For $\alpha \in k$ and $s>0$, let $D:=\{x \in \mathbb{P}_k^{1, \mathrm{an}}: |T- \alpha|_x<s\}$ be the open disc centered at $\alpha$ and of radius $s$.
Let $\beta, \gamma \in k$. 

(1) If, for some $r>0,$ $\eta_{\beta, r} \in D$, then $\eta_{\beta,0} \in D$. For any two rigid points $\eta_{\beta, 0}, \eta_{\gamma,0} \in D$, $|\beta-\gamma|<s$. 

(2) If $\eta_{\beta, 0} \in D$, or equivalently, $|\alpha-\beta|<s$, then $D=\{x \in \mathbb{P}_{K}^{1, \mathrm{an}}: |T-\beta|_x <s\}.$ 
\end{lm}

\begin{proof}
(1) If $\eta_{\beta, r} \in D$, then $|T-\alpha|_{\eta_{\beta, r}}=\max(|\alpha-\beta|, r)<s$, so ${|T-\alpha|_{\eta_{\beta, 0}}=|\beta-\alpha|<s}$, implying $\eta_{\beta, 0} \in D$.  
 Similarly, if $\eta_{\gamma,0} \in D$, $|T-\alpha|_{\eta_{\gamma, 0}}=|\alpha-\gamma|<s$, so $|\beta-\gamma|\leqslant \max(|\alpha-\beta|, |\alpha-\gamma|)<s$. 

(2) Let $x \in D,$ meaning $|T-\alpha|_x < 
s$. Then $|T-\beta|_x \leqslant \max(|T-
\alpha|_x, |\alpha-\beta|) <s$. Similarly, if $y \in \mathbb{P}_{k}^{1, \mathrm{an}}$ 
satisfies $|T-\beta|_y < s$, then $|T-\alpha|_y<s$. 
\end{proof}

\begin{rem} \label{5.2.1}
For any $b \in \widehat{\overline{k}}$, there exists a canonical retraction ${d_b: \mathbb{P}_{\widehat{\overline{k}}}^{1, \mathrm{an}} \rightarrow \Gamma_b}$ such that $|T-b|_{u}=|T-b|_{d_b(u)}$ for any $u \in \mathbb{P}_{\widehat{\overline{k}}}^{1, \mathrm{an}}$ (\emph{cf.} \cite[3.4.24]{duc}). Moreover, if $u, u' \in \Gamma_b$, $u \neq u'$, $|T-b|_u \neq |T-b|_{u'}.$ Here $\Gamma_b:=[\eta_{b,0}, \infty]$ is the injective path in $\mathbb{P}_{\widehat{\overline{k}}}^{1, \mathrm{an}}$ connecting $\eta_{b,0}$ and~$\infty$. 
Thus, $T-b$ is strictly increasing on $\Gamma_b$ and locally constant elsewhere.
\end{rem}

Recall in \cite[pg. 210]{duc} the notions of \emph{virtual discs} and \emph{virtual annuli}. 

\begin{rem} \label{intro}
Let $X/k$ be a variety. Let $L/k$ be an open virtual disc or open virtual annulus that can be embedded in $\mathbb{P}_k^{1, \mathrm{an}}$. Let $y \in L$ be such that $X(\mathscr{M}_y) \neq \emptyset.$ 
We construct automorphisms of $L$ which send a random point $x$ to a point like $y$, with the purpose of then obtaining that $X(\mathscr{M}_{x}) \neq \emptyset.$
\end{rem}

\begin{prop} \label{5.1}
Let $L/k$ be an open virtual disc embedded in $\mathbb{P}_k^{1, \mathrm{an}}$. Let $X/k$ be a variety. Assume there exists an open neighborhood $U$ in $L$ of the end $\omega$ of $L$ such that $X(\mathscr{M}(U)) \neq \emptyset$. Then for any $m \in \mathbb{N}$, there exists a finite field extension $l/k$ such that $([l:k], m)=1$ and $X(\mathscr{M}_{L_l, x}) \neq \emptyset$ for all $x \in L_l$. 

If $|k^{\times}|$ is dense in $\mathbb{R}_{>0}$, then one can take $l:=k$. 
\end{prop}

\begin{proof}
Fix $m \in \mathbb{N}$. Let $p: \mathbb{P}_{\widehat{\overline{k}}}^{1, \mathrm{an}} \rightarrow \mathbb{P}_k^{1, \mathrm{an}}$ be the projection morphism. From \cite[Cor.~1.3.6]{Ber90}, $G:=\mathrm{Gal}(\overline{k}/k)$ acts on $\mathbb{P}^{1, \mathrm{an}}_{\widehat{\overline{k}}}$ in a way that $p$ induces an isomorphism $\mathbb{P}_{\widehat{\overline{k}}}^{1, \mathrm{an}}/G \cong \mathbb{P}_k^{1, \mathrm{an}}$. As~$L$ is an open virtual disc embedded in $\mathbb{P}_k^{1, \mathrm{an}}$, the preimage $p^{-1}(L)$ is a finite disjoint union $\bigsqcup_{i \in I} D_i$ of open discs~$D_i$. By \emph{loc.cit.}, the restriction of $p$ to $L_{\widehat{\overline{k}}}=p^{-1}(L)$, which we still denote $p$, induces an isomorphism $L_{\widehat{\overline{k}}}/G \cong L$. 

In order to parse the proof, we break it down into four different parts. 

\noindent \textbf{Part I: Lifting to $\widehat{\overline{k}}$.} 
For $i \in I$, let $\alpha_i \in \widehat{\overline{k}}$ and $s_i \in \mathbb{R}_{>0}$ be such that $D_i=\{x \in \mathbb{P}_{\widehat{\overline{k}}}^{1, \mathrm{an}} :|T-\alpha_i|_x<s_i\}.$ The action of $G$ on $\bigsqcup_{i \in I}D_i$, which permutes the $D_i$, implies that for any $i', i'' \in I$, $s_{i'}=s_{i''}=:s$. The set $p^{-1}(\omega)$ consists of the unique boundary points $\omega_i$ of $D_i$, and $\omega_i=\eta_{\alpha_i, s}$ for all $i \in I$.
As $U$ is an open neighborhood of $\omega$, $p^{-1}(U)$ is a disjoint union of open neighborhoods $U_i$ of $\omega_i$ in $D_i$, $i \in I$. As such, we may assume $\partial_{D_i}{U_i}$ is a finite set of points of type~2 or~3 of $D_i$.
 From \cite[Prop. 2.3]{une}, for $z \in \partial_{D_i}{U_i}$, there exists $\alpha_{z} \in \widehat{\overline{k}}$ and $r_{z}>0$ such that $z=\eta_{\alpha_z, r_z}$.

For $z \in \partial_{D_{i}} U_{i}$, let $\gamma_z$ denote the unique path in $D_{i}$ connecting the point $\eta_{\alpha_z, 0}$ to $\omega_{i}$, meaning $\gamma_z=\{\eta_{\alpha_z, c} \in D_{i}: c>0\}$ (see \cite[Rem. 1.8.26]{doktoratura}). We recall that it is homeomorphic to the open interval $(0, s)$. As $U_{i}$ is connected, $\gamma_z \cap U_{i}$ is a connected subset of $\gamma_z$. As $U_{i}$ is an open neighborhood of $\omega_{i}$, we obtain that $\gamma_z \cap U_i=\{\eta_{\alpha_z, c} \in D_{i}: c>r_z\} \subseteq U_{i}$, and for any $t \leqslant r_z, \eta_{\alpha_z, t} \notin U_{i}$. For any two points $z, z' \in \partial_{D_i} U_i$, \begin{equation} \label{5.3.1}
r_z < |\alpha_z-\alpha_{z'}|.
\end{equation} 
To see this, suppose on the contrary that $r_z \geqslant |\alpha_z-\alpha_{z'}|$. Then $\eta_{\alpha_z, r_z}=\eta_{\alpha_{z'}, r_z}$. As $z \neq z'$, $r_z \neq r_{z'}.$ If $r_{z}< r_{z'},$ then $\eta_{\alpha_{z'}, r_{z'}}=\eta_{\alpha_z, r_{z'}}$, meaning $\eta_{\alpha_{z'}, r_{z'}} \in \gamma_z \cap U_i \subseteq U_i$, contradiction. Similarly, if $r_{z'}< r_{z}$, then $\eta_{\alpha_z, r_z}=\eta_{\alpha_{z'}, r_z} \in \gamma_{z'} \cap U_i \subseteq U_i$, contradiction. 

\

\noindent \textbf{Part II: Construction of $l$.}
By Lemma \ref{5.2}(1), for any  $z, z' \in \partial_{D_{i}}{U_{i}}$, $|\alpha_z-\alpha_{z'}| <s$.
Let $a,b \in \mathbb{R}_{>0}$ be such that $\max_{i \in I}\max_{z,z' \in \partial_{D_{i}}{U_{i}}}(|\alpha_z-\alpha_{z'}|) <a<b <s$. 

\begin{lm} \label{extensions}
There exists a finite field extension $l/k$ such that $([l:k], m)=1$ and $ w \in l$ with $|w|=:r \in (a,b)$ (with respect to the unique norm on $l$ extending that of $k$).
If $|k^{\times}|$ is dense in $\mathbb{R}_{>0}$, then one can take $l=k$.
\end{lm}

\begin{proof}
The statement is immediate if $|k^{\times}|$ is dense in $\mathbb{R}_{>0}$. Otherwise, $k$ is a discretely valued field, and let $\pi$ be a uniformizer. Seeing as the divisible closure $\sqrt{|k^{\times}|}$ of the value group $|k^{\times}|$ is dense in $\mathbb{R}_{>0}$, there exists a large enough integer $h$ such that $(h, m)=1$ and $(a,b) \cap \sqrt[h]{|k^{\times}|}$, where $\sqrt[h]{|k^{\times}|}:=\{r \in \mathbb{R}_{>0}: r^h \in |k^{\times}|\}$. Let $r \in (a,b) \cap \sqrt[h]{|k^{\times}|}$, meaning there exists $n \in \mathbb{Z}$ such that  $r^h=|\pi|^n$. 
Set $P(X):=X^h-\pi \in k[X]$. By Eisenstein's criterion of irreducibility, $P(X)$ is an irreducible polynomial over $k$. Set $l:=k[X]/(P(X))$. Then $[l:k]=h$. Clearly, $l$ contains a root $\alpha$ of $P(X)$, implying that $|\alpha|=|\pi|^{1/h}$. As a consequence, $r \in |l^{\times}|$, meaning there exists $w \in l$ such that $|w|=r \in (a,b)$.      
\end{proof}

From now on, let $l/k$, $w \in l$ and $r$ be as in Lemma \ref{extensions}. A base change of the isomorphism $l[T] \rightarrow l[T], T \mapsto T+w$, induces an isomorphism $\psi: \mathbb{P}_{\widehat{\overline{k}}}^{1, \mathrm{an}} \rightarrow \mathbb{P}_{\widehat{\overline{k}}}^{1, \mathrm{an}}$, and $\eta_{\gamma, c} \mapsto \eta_{\gamma+w, c}$ for any $\gamma \in \widehat{\overline{k}}$ and $c \geqslant 0$. For $i \in I$, $|T-\alpha_i|_x<{s}$ if and only if $|T-\alpha_i+w|_{\psi(x)}<{s}.$ As $|\alpha_i-(\alpha_i-w)|=|w|<s$, by Lemma~\ref{5.2}(2), $\psi(D_i)=D_i$. Thus $\psi$, \emph{i.e.} the map $T \mapsto T+w$, induces an isomorphism $D_i \rightarrow D_i$ for $i \in I$, meaning an isomorphism $L_{\widehat{\overline{k}}} \rightarrow L_{\widehat{\overline{k}}}$, which we still denote $\psi.$  

\

\noindent \textbf{Part III: The set $\psi((L \backslash U)_{\widehat{\overline{k}}})$.} Let $x \in L \backslash U$ and $x' \in p^{-1}(x)$. There exists $j \in I$ such that $x' \in D_j$ and as $x \notin U, $ we have $x' \notin U_j.$ Let us show that \begin{equation} \label{5.5}
\psi(x') \in U_j.
\end{equation} 
Let $[x', \omega_j) \subseteq D_j$ be the unique injective path connecting $x'$ to $\omega_j$. As $x' \notin U_j$, $[x', \omega_j)$ intersects $\partial_{D_j} U_j$ at a single point $\eta_{\alpha_z, r_z}.$ Then $[x', \omega_j)=[x', \eta_{\alpha_{z}, r_z}] \sqcup (\eta_{\alpha_z, r_z}, \omega_j)$.
By Lemma \ref{5.2}(2), 
$D_j=\{x \in \mathbb{P}_{\widehat{\overline{k}}}^{1, \mathrm{an}}: |T-\alpha_z|_x <s\}$. Its boundary point $\omega_j$ coincides with $\eta_{\alpha_z, s},$ so $[x', \omega_j)=[x', \eta_{\alpha_z, r_z}) \sqcup [\eta_{\alpha_z, r_z}, \eta_{\alpha_z,s})$. By Remark \ref{5.2.1}, $|T-\alpha_z|_{x'} \leqslant r_z <r.$ Then
 $|T-\alpha_z+w|_{\psi(x')}=|T-\alpha_z|_{x'} < r$.

Let us show that
 $\{y \in D_{j}:|T-\alpha_z+w|_y < r\} \subseteq U_{j},$ thus implying  $\psi(x') \in U_j$. 
Suppose there exists $y \in D_i$ such that $|T-\alpha_z+w|_y < r$, but $y \notin U_j$.
We remark that $|T-\alpha_z|_y=\max(|T-\alpha_z+w|_y, |w|)=r,$ implying $d_{\alpha_z}(y)=\eta_{\alpha_z, r}$ (see Remark \ref{5.2.1}). If ${y \notin U_j}$, then the unique injective path $[y, \eta_{\alpha_z, r})$ in $D_j$ connecting $y$ and $\eta_{\alpha_z,r}$  intersects $\partial_{D_j}U_j$ at a single point $\eta_{\alpha_{z'}, r_{z'}}$. As ${d_{\alpha_z}(y)=d_{\alpha_z}(\eta_{\alpha_{z},r})=\eta_{\alpha_z, r}}$ and $\eta_{\alpha_{z'}, r_{z'}} \in [y, \eta_{\alpha_z, r})$, by Remark~\ref{5.2.1}, $d_{\alpha_z}(\eta_{\alpha_{z'}, r_{z'}})=\eta_{\alpha_z, r}$. Consequently, $|T-\alpha_z|_{\eta_{\alpha_{z'}, r_{z'}}}=r.$ At the same time, $|T-\alpha_z|_{\eta_{\alpha_{z'}, r_{z'}}}=\max(|\alpha_z-\alpha_{z'}|,r_{z'})$, so by (\ref{5.3.1}), $|T-\alpha_z|_{\eta_{\alpha_{z'}, r_{z'}}}=|\alpha_z-\alpha_{z'}|.$ Thus, $r=|\alpha_z-\alpha_{z'}|$, which is in contradiction with the choice of $r$ in Lemma \ref{extensions}.

\

\noindent \textbf{Part IV: Conclusion.}
Let $\pi_l$ be the projection $\pi_l: L_l \rightarrow L$. Set $G'=\mathrm{Gal}(\overline{k}/l)$. As $\psi: \mathbb{P}^{1, \mathrm{an}}_{\widehat{\overline{k}}} \rightarrow \mathbb{P}^{1, \mathrm{an}}_{\widehat{\overline{k}}}$ is defined over $l$, it is $G'$-equivariant, so it induces a $G'$-equivariant isomorphism $L_{\widehat{\overline{k}}} \rightarrow L_{\widehat{\overline{k}}}$. Consequently, we obtain an isomorphism $\varphi: L_{\widehat{\overline{k}}}/G' \rightarrow L_{\widehat{\overline{k}}}/G'$. From \cite[Cor. 1.3.6]{Ber90}, $L_{\widehat{\overline{k}}}/G' \cong L_l$, so $\varphi: L_l \xrightarrow{\sim} L_l$, and the base change to $\widehat{\overline{k}}$ induces $\psi$. The  diagram (\ref{njeshi}) is commutative, where $q$ is the projection $L_{\widehat{\overline{k}}} \rightarrow L_l$.
\begin{equation} \label{njeshi}
\begin{tikzcd}
L_{\widehat{\overline{k}}} \arrow{r}{\psi} \arrow{d}[swap]{q}& L_{\widehat{\overline{k}}} \arrow{d}{q} \\
L_l  \arrow{r}[swap]{\varphi} & L_l
\end{tikzcd}
\end{equation}
Let $x \in L_l \backslash U_l$. For any $y \in q^{-1}(x)$, $y \not \in q^{-1}(U_l)=\bigcup_{i \in I} U_i$. Let $j \in I$ such that $y \in D_j$. By (\ref{5.5}) above, $\psi(y) \in U_j,$ so $q(\psi(y)) \in U_l$. As $q(\psi(y))=\varphi(q(y))=\varphi(x)$, we obtain $\varphi(x) \in U_l$. The isomorphism $\varphi$ induces an isomorphism of fields $\mathscr{M}_{L_l, x} \cong \mathscr{M}_{L_l, \varphi(x)}$ (which fixes $l$ but not $F$). As $\varphi(x) \in U_l$, $\pi_l(\varphi(x))) \in U,$ so $\mathscr{M}_L(U) \subseteq \mathscr{M}_{L, \pi_l(\varphi(x))} \subseteq \mathscr{M}_{L_l, \varphi(x)}$. By assumption, $X(\mathscr{M}_L(U)) \neq \emptyset$, so $X(\mathscr{M}_{L_l, \varphi(x)}) \neq \emptyset$. As $\mathscr{M}_{L_l, \varphi(x)} \cong \mathscr{M}_{L_l,x},$ we obtain $X(\mathscr{M}_{L_l, x}) \neq \emptyset.$  
If $x \in U_l$, then $\mathscr{M}_{L}(U) \subseteq \mathscr{M}_{L, \pi_l(x)} \subseteq \mathscr{M}_{L_l, x}$, implying $X(\mathscr{M}_{L_l, x}) \neq \emptyset,$ thus concluding the proof of Proposition \ref{5.1}.
\end{proof}

\begin{prop} \label{5.5.1}
Let $L/k$ be an open virtual annulus that is embedded in $\mathbb{P}_k^{1, \mathrm{an}}$. Let us denote its ends by $\omega_1$ and $\omega_2$. Let $X/k$ be a variety. Assume that there exists an open neighborhood $U$ in $L$ of the end $\omega_1$ of $L$ such that $X(\mathscr{M}(U)) \neq \emptyset$. Then for any $m \in \mathbb{N}$, there exists a finite field extension $l/k$ such that $([l:k], m)=1$ and $X(\mathscr{M}_{L_l, x}) \neq \emptyset$ for all $x \in L_l$.  
If $|k^{\times}|$ is dense in $\mathbb{R}_{>0}$, then one can take $l:=k$.
\end{prop}

\begin{proof}
Let $p: \mathbb{P}_{\widehat{\overline{k}}}^{1, \mathrm{an}} \rightarrow \mathbb{P}_k^{1, \mathrm{an}}$ denote the projection. Then $L_{\widehat{\overline{k}}}=p^{-1}(L)$ is a finite disjoint union $\bigsqcup_{i \in I} L_i$ of open annuli $L_i$.  For $i \in I$, there exist $\alpha_i \in \widehat{\overline{k}}$ and $r_i, s_i \in \mathbb{R}_{>0}$ such that $L_i=\{x \in \mathbb{P}_{\widehat{\overline{k}}}^{1, \mathrm{an}}: r_i < |T-\alpha_i|_x <s_i\}$. As the action of $G:=\mathrm{Gal}(\overline{k}/k)$ on $\bigsqcup_{i \in I}L_i$ permutes the $L_i$, for any $i', i'' \in I$, $r_{i'}=r_{i''}=:r$ and $s_{i'}=s_{i''}=:s$. The ends of $L_i$ are the points $\omega_{1,i}:=\eta_{\alpha_i, s}$ and $\omega_{2,i}:=\eta_{\alpha_i,r}$. The action of $G$ on $L_{\widehat{\overline{k}}}$ permutes the sets $\{\omega_{1,i}\}_{i \in I}$  and $\{\omega_{2,i}\}_{i \in I}$. Consequently, $p(\{\omega_{1,i}\}_{i \in I})$ is a single point which is also an end of $L$; let us assume it is the point $\omega_1$ (by a change of coordinate on $\mathbb{P}_{\widehat{\overline{k}}}^{1, \mathrm{an}}$ if necessary). Similarly, $p(\{\omega_{2,i}\}_{i \in I})=\omega_2$.  

For any $i \in I$, let $D_i:=\{x \in \mathbb{P}_{\widehat{\overline{k}}}^{1, \mathrm{an}}: |T-\alpha_i|_x <s\}$. This is an open disc satisfying $L_i \subseteq D_i$. If $i' \neq i''$, then $D_{i'} \cap D_{i''} =\emptyset.$ Otherwise, as $D_{i'} \cap D_{i''}$ is open, it contains a rigid point $\eta_{\gamma, 0}$ for some $\gamma \in \overline{k}$. By Lemma \ref{5.2}(2), $D_{i'}=D_{i''}$.
As they have a common end, $\partial{D_{i'}}=\partial{D_{i''}}=\{\eta_{\alpha_{i'},s}\}$, so
 $L_{i'} \cap L_{i''} \neq \emptyset$, contradiction.    Moreover, $G$ acts on $\bigsqcup_{i \in I}D_i$, meaning $D:=p(D_i)=p(\bigcup_{i \in I}D_i)$ is an open virtual disc defined over $k$ and with end~$\omega_1$. By construction, $L \subseteq D$. 
Let us fix the integer $m$. By Proposition \ref{5.1}, there exists a field extension $l/k$ such that $([l:k], m)=1$ and for any $x \in D_{l}:=D \times_k l$, we have $X(\mathscr{M}_{D_l, x}) \neq \emptyset$. Since $L_l \subseteq D_l$, the proof is concluded. 
\end{proof}

\begin{rem} \label{5.5.2}
We recall that if $S$ is a \emph{triangulation} (\emph{cf.} \cite[5.1.13]{duc}) of a $k$-analytic curve~$C$, then the connected components of $C \backslash S$ are open virtual discs and open virtual annuli. There is, \emph{a priori}, no embedding of an arbitrary open virtual disc or an open virtual annulus into~$\mathbb{P}^{1, \mathrm{an}}.$

We also recall that a smooth analytic curve always has a triangulation (\cite[Thm.~5.1.14]{duc}).
\end{rem} 

\begin{thm} \label{5.6} Let $C/k$ be a proper normal geometrically connected and smooth analytic curve. Set $F:=\mathscr{M}(C)$.
Assume $C$ has a  triangulation $S$ such that all of the connected components of $C \backslash S$ can be embedded in $\mathbb{P}_k^{1, \mathrm{an}}$. Let $X$ be a $k$-variety. Suppose there exists a rational linear algebraic group $G/F$ acting strongly transitively on~$X_F$. 
If $X(\mathscr{M}_{s}) \neq \emptyset$ for all $s \in S$, and 
\begin{enumerate}
\item $|k^{\times}|$ is dense in $\mathbb{R}_{>0}$, then $X(F) \neq \emptyset$;
\item $k$ is discretely valued, then $X$ has a zero cycle of degree one over~$F$.
\end{enumerate}
\end{thm}

\begin{proof} \begin{sloppypar} Let $s \in S$.
As $X(\mathscr{M}_{s}) \neq \emptyset$, there exists a neighborhood $U_s$ of $s$ in $C$ such that ${X(\mathscr{M}(U_s)) \neq \emptyset}$. We may assume that $\partial{U_s}$ is a finite set of points of type 2 and 3. The set $C \backslash S$ is a disjoint union of open virtual discs and open virtual annuli. Since $\bigcup_{s\in S}\partial{U_s}$ is finite, there are only finitely many connected components of $C \backslash S$ not entirely contained in $\bigcup_{s\in S} U_s$ (\emph{cf.} Remark~\ref{1.11}). Let us denote them by $L_1,L_2, \dots, L_n.$ If $|k^{\times}|$ is dense in $\mathbb{R}_{>0}$, then by Propositions \ref{5.1} and \ref{5.5.1}, $X(\mathscr{M}_{C,x}) \neq \emptyset$ for all $x \in C$, so from \cite[Thm.~3.11]{une}, $X(F) \neq \emptyset.$ 
\end{sloppypar}
 \begin{sloppypar}
Otherwise, By Propositions \ref{5.1} and \ref{5.5.1}, there exists a finite extension $l_1/k$ for which ${X(\mathscr{M}_{L_{1,l_1},x}) \neq \emptyset}$ for all $x \in L_{1, l_1}$. Set $[l_1:k]=m_1$. For $i \in \{2,3,\dots, n\},$ let $l_i/k$ be a finite extension such that $([l_i:k], \Pi_{j=1}^{i-1} m_j)=1$, where $m_{j}:=[l_{j}:k]$, and $X(\mathscr{M}_{L_{i,l_i},x}) \neq \emptyset$ for all $x \in L_{i, l_i}$.
Let $l/k$ be the composite of $l_1, l_2, \dots, l_n$ in $\overline{k}$. Then $\mathscr{M}(C_l)=F \otimes_k l=:E$. Also ${[l:k]=\Pi_{i=1}^n [l_i:k]}=m_1m_2\cdots m_n=:m,$ and as $C$ is geometrically connected, $[E:F]=[l:k]=m$. Let $p$ be the projection $C_l \rightarrow C$. For $x \in C \backslash \bigcup_{i=1}^n L_i$, $X(\mathscr{M}_{C,x}) \neq \emptyset$, so for $x \in C_l \backslash \bigcup_{i=1}^n L_{i,l}$, we obtain $X(\mathscr{M}_{C_l,x}) \neq \emptyset$. For  $x \in \bigcup_{i=1}^n L_{i,l}$, by construction, $\mathscr{M}_{L_{i,l_i},x}=\mathscr{M}_{C_{l_i}, x} \subseteq \mathscr{M}_{C_l,x}$ for all $i$, so $X(\mathscr{M}_{C_l},x) \neq \emptyset$. From \cite[Thm~3.11]{une},~${X(\mathscr{M}(C_l)) \neq \emptyset}$. 
\end{sloppypar}

By Propositions \ref{5.1} and \ref{5.5.1}, there exists a finite extension $l_1'/k$ with ${([l_1',k],m)=1}$, and $X(\mathscr{M}_{L_{1, l_1'},x}) \neq \emptyset$ for all $x \in L_{1, l_1'}.$ Set $m_1':=[l_1':k].$ For $i \in \{2,3,\dots, n\},$ let $l_i'/k$ be a finite field extension such that $([l_i':k], m\Pi_{j=1}^{i-1} m_j')=1$, where $m_{j}':=[l_{j}':k]$, and $X(\mathscr{M}_{L_{i,l_i'},x}) \neq \emptyset$ for all $x \in L_{i, l_i'}$. 
Let $l'/k$ be the composite of $l_1', l_2', \dots, l_n'$ in $\overline{k}$. Set $m':=[l':k]$. Then $(m', m)=1$ and as in the case of $l$, $X(\mathscr{M}(C_{l'})) \neq \emptyset$ and $E':=\mathscr{M}_{C_{l'}}(C_{l'})$ satisfies $[E':F]=m'$.
Thus, $X(E) \neq \emptyset$ and $X(E') \neq \emptyset$, where $([E:F], [E':F])=1$, so $X$ has a zero cycle of degree one over $F$.    
\end{proof}

\begin{rem} \label{5.7}
In the proof of Theorem \ref{5.6}, we only use that a \emph{finite} number of the connected components of $C \backslash S$ can be embedded in $\mathbb{P}_k^{1, \mathrm{an}}$. They depend on the elements of $X(\mathscr{M}_{s}), s\in S$.  Recall that $X(\mathscr{M}_s) \neq \emptyset$ is equivalent to $X(F_{v_s}) \neq \emptyset$, where $v_s:=val(s)$ (\emph{cf.} relation~\eqref{artin}).

 The case $C=\mathbb{P}_k^{1, \mathrm{an}}$ satisfies trivially the assumptions of Theorem \ref{5.6} with respect to any triangulation or vertex set. From \cite[4.4]{lordan}, if $S$ is a vertex set of $C$ corresponding to a \emph{semi-stable} model $\mathscr{C}$ of $C^{\mathrm{al}}$, then the assumptions of Theorem~\ref{5.6} are satisfied. Another example is given by Mumford curves, which can be locally embedded in~$\mathbb{A}_k^{1, \mathrm{an}}$.  
\end{rem}

\begin{rem} \label{5.8}
There are several families of varieties for which the existence of zero cycles of degree one implies the existence of  rational points (\emph{e.g.} see \cite{parimala}). In particular, this is true for quadratic forms (\emph{cf.} \cite[VI, Prop. 1.9]{lam} and \cite{gillen}). 
\end{rem}

\begin{cor} \label{5.9} 
With the notation of Theorem \ref{5.6}, let $q$ be a quadratic form defined over~$k$. If $q$ is isotropic over $\mathscr{M}_{s}$ for all $s \in S$,  then $q$ is isotropic over ~$F$.   
\end{cor}

\section{Appendix: other examples} \label{appendix}

In Section \ref{quad}, we used a theorem of Springer (\cite[VI, Prop. 1.9]{lam}) on quadratic forms to reduce to a case where the results of Section~\ref{section1} are applicable, and thus obtained Theorem~\ref{4.5}--a Hasse principle for quadrics. Springer's theorem allows us to reduce to quadrics which satisfy some strong smoothness assumptions over models of the curve.
By the same principle, the results of Section \ref{section1} should apply to any variety for which we can show a Springer-type theorem.
\subsection{Unitary groups (\emph{cf.} \cite{wu} and \cite[Sect. 12]{parsur})} \label{unitary}

In \cite{larmour}, Larmour showed a Springer-type theorem for Hermitian forms. By mimicking the proof of Theorem~\ref{4.5} and using the results of Section \ref{section1}, this gives rise to a Hasse principle for homogeneous varieties under unitary groups. Hermitian forms are defined over division algebras, so a natural starting point is studying whether valuations on a field $K$ (for us, $K=\mathrm{Frac} \ \mathcal{O}^{\circ}(U_x)$, see proof of Thorem~\ref{4.5}) extend ``well'' to  division algebras over~$K$.

This has been studied in \cite{wu} under some restrictions. The results were then generalized using the same approach in \cite{parsur}. In both cases, the authors show a Hasse principle for homogeneous varieties under unitary groups. We give here a brief summary of the interpretation of these results in our setting, without claim to originality. In a \cite{guhsur}, the authors generalize the main Theorem \ref{7.5} in some special cases. 

For a detailed account on the notions of this section, we refer the reader to \cite{involutions}. We will use the notation from Section~\ref{1.1}. Assume, moreover, that $k$ is a local field and ${\mathrm{char} \ k \neq 2}.$

\subsubsection{Homogeneous spaces over unitary groups}  \label{subsection7.1}

Let $L:=F(\sqrt{\lambda})$ be a quadratic field extension over $F$. Let $A$ be a central simple algebra over $L$ and $D$ the division algebra Brauer equivalent to it. Let $\sigma$ be an involution of the second kind on $A$ such that $L^{\sigma}=F$. Let $(V,h)$ be a Hermitian form over $(A, \sigma)$. Then $\sigma$ induces uniquely an involution of second kind $\tau$ on $D$, and $(V,h)$ induces uniquely a Hermitian form $(V', h')$ on $(D, \tau)$. We will denote by $G:=\textit{U}(A, \sigma, h)=\textit{U}(D, \tau, h')$ the associated unitary group defined over~$F$.  From \cite[Prop.~2.4]{merkchern}, $G$ is a rational reductive linear algebraic group. The form $h$ is \emph{isotropic over $F$} if there exists $v \in V \backslash \{0\}$ such that $h(v,v)=0$; a subspace $W \subseteq V$ is \emph{totally isotropic} over $(V,h)$ if for all $w \in W$, $h(w,w)=0$.

\begin{defn} \label{7.2}  Let $0<n_1 < n_2< \cdots < n_r  < \deg{A}$ be integers; $\deg{A}$ denotes the degree of $A$. Let $X_h(n_1, n_2, \dots, n_r)$ be the projective $F$-variety such that for any field extension $K/F$, 
$$X_h(n_1, n_2, \dots, n_r)(K)=\{(W_1, W_2, \dots, W_r) : $$$$ 0 \subsetneq W_1 \subsetneq W_2 \cdots \subsetneq W_r \subseteq V_K, W_i \ \text{is totally isotropic}, \ \dim_L{W_i}=n_i\cdot \deg{A_K}  \ \forall i\},$$
where $A_K:=A \otimes_F K, V_K:=V \otimes_{F} K.$
\end{defn}

\begin{thm}[{\cite[\S 2]{wu}, \cite[Thm. 12.1]{parsur}}] \label{7.3}

(1) For any projective homogeneous $F$-variety~$X$ over $G$ there exists a sequence of integers $0< n_1 < n_2 < \cdots < n_r \leqslant \deg{A}/2$ such that $X \cong X_h(n_1, n_2, \dots, n_r).$

(2) Let $K/F$ be a field extension. Then $X_h(n_1, n_2, \dots, n_r)(K) \neq \emptyset$ if and only if $X_h(n_r)(K) \neq \emptyset$, and $\mathrm{ind}(A)$ divides $n_i$ for all $i,$ where $\mathrm{ind}(A)$ is the index of $A$.

(3) There exists a bijection $X \mapsto X_0$ between the projective homogeneous $F$-varieties over $\textit{U}(A, \sigma, h)$ and the projective homogeneous $F$-varieties over $\textit{U}(D, \tau, h')$. Moreover, for any field extension~$K/F$, $X(K) \neq \emptyset \iff X_0(K) \neq \emptyset.$
\end{thm} 

\begin{rem} \label{7.3.1} Similar results exist if $\sigma$ is an involution of the first kind and $G=\textit{SU}(A, \sigma, h)$-the special orthogonal group (see \cite[Sect. 2]{wu}). See also Subsection \ref{subsection7.2}. 
\end{rem}
 
We haven't yet used the hypothesis that $k$ is local. It is necessary from this point on. 
 
\begin{prop}\label{7.4} \begin{sloppypar} Let $X:=X_h(n_1, n_2, \dots, n_r)$ be a homogeneous space over ${G=U(A, \sigma, h)}$.
Assume ${X(F_v) \neq \emptyset}$ for all $v \in V(F)$ discrete. Then $\mathrm{ind}(A)$ divides~$n_i$ for all $i \in \{1, 2, \dots, r\}$.
\end{sloppypar}
\end{prop}

\begin{proof}
By Theorem \ref{7.3}(2), for all $v \in V(F)$ discrete, $\mathrm{ind}(A \otimes_F F_v)$ divides $n_i$ for all~$i$. From \cite[Thm. 5.5]{hhk} and \cite[Prop. 5.10]{parsurpreeti} (see also first paragraph of proof for \cite[Thm.~12.2]{parsur}), $\mathrm{ind}(A)$ also divides $n_i$ for all $i.$ 
\end{proof}

\subsubsection{Maximal orders on division algebras (\emph{cf.}  \cite[Sect. 10]{parsur})}
Let $R$ be a complete regular local ring of dimension $2$ (for us, $R=\mathcal{O}^{\circ}(U_x)$, see Section \ref{section1}). Set $K:=\mathrm{Frac}(R)$. Let $\widetilde{K}$ be its residue field; assume it is finite and that $\mathrm{char} \ \widetilde{K} \neq 2$.  
Let $\omega, \delta$ be generators of the maximal ideal of $R$. The rings $R_{(\omega)}$ and $R_{(\delta)}$ are discretely valued. We denote by~$\widehat{R_{(\omega)}}$, resp. $\widehat{R_{(\delta)}}$, their completions.

Let $\lambda \in R$ be such that either $\lambda=u$ or $\lambda=u\omega$ for $u \in R^{\times}$. Set $L:=K(\sqrt{\lambda})$. Let $S$ denote the integral closure of $R$ in $L$. Then $S$ is also a regular local ring of dimension 2. Moreover, $\omega_1, \delta$ generate the maximal ideal of $S$, where $\omega_1:=\omega$ if $\lambda=u$ and $\omega_1:=\sqrt{u\omega}$ otherwise. 
Let $D$ be a division algebra over $L$; we denote its reduced norm by $Nrd_D$. Let~$\tau$ be an involution on~$L$ of the second kind such that $L^{\tau}=K$. 
Assume $D$ is \emph{unramified} (see \cite[pg. 4]{parsur}) on $S$ except possibly at $(\omega_1)$ and $(\delta)$. 

\begin{thm}[{\cite[Lemma 3.7]{wu}, \cite[Lemma 9.3]{parsur}}] \label{7.4.1} Assume that $(\deg{D}, \mathrm{char} \ \widetilde{K})=1$. There exists an $R$-maximal order $\Lambda$ in~$D$ such that 
$\tau(\Lambda)=\Lambda$, and
\begin{enumerate}
\item there exist $\omega_D$, $\delta_D \in \Lambda$ such that $\tau(\omega_D)=\omega$ and $\tau(\delta_D)=\delta$;
\item if $a \in \Lambda$ is such that $\tau(a)=a$ and $Nrd_D(a)=u \omega^r \delta^s$ for $u \in R^{\times}, r,s \in \mathbb{Z}_{\geqslant 0}$, then, up to a square factor in $\Lambda$, $a=u \omega_D^{\epsilon_1} \delta_D^{\epsilon_2},$ where $\{\epsilon_1, \epsilon_2\} \in \{0,1\}$, with same parity as $r$ and $s$, respectively.

\end{enumerate}
\end{thm}

\begin{rem}{\label{7.4.2}} If $h=<a_1, a_2, \dots, a_n>$ is a Hermitian form over $(D, \tau)$, then $\tau(a_i)=a_i$ for all $i$. By Theorem \ref{7.4.1}, $h \cong h_1 \bot \omega_D h_2 \bot \delta_D h_3 \bot \omega_D \delta_D h_4$, where $h_j=<b_{1}^j, \cdots, b_{m_j}^j>$ with $b_{s}^j \in \Lambda^{\times}$ for all $s$ and all $j$.
\end{rem}
The following is a consequence of the result of Larmour (\cite{larmour}) for Hermitian forms. It plays the same role as the Theorem of Springer for quadratic forms.
\begin{thm}[{\cite[Cor. 9.5]{parsur}}] \label{7.4.3}
Let $h$ be a Hermitian form over $(D, \tau)$. If $h$ is isotropic over $\mathrm{Frac}(\widehat{R_{(\omega)}})$ or $\mathrm{Frac}(\widehat{R_{(\delta)}})$, then it is isotropic over $K$. 
\end{thm}

\begin{rem}\label{evertete?} 
Thanks to Theorem~\ref{7.3}(2), the existence of rational points on $X:=X_h(n_1, n_2, \dots, n_r)$ is equivalent to questions of isotropy of the Hermitian form $h$. By Remark~\ref{7.4.2} and Theorem~\ref{7.4.3}, we can reduce to particular Hermitian forms $h$ of the type $<a_1, \dots, a_n>$, with~${a_i \in \Lambda^{\times}}$. The corresponding variety $X$ will then satisfy the smoothness hypotheses of Section \ref{section1}.
\end{rem}

\subsubsection{The Hasse principle for $G$} We continue using the same notation as in Subsect.~\ref{subsection7.1}. 
 Let $X/F$ be a projective homogeneous variety under $G=\textit{U}(A, \sigma, h)$.  

\begin{thm}[{\cite[Thm. 4.4]{wu}, \cite[Thm. 12.3]{parsur}}] \label{7.5} Suppose $(2 \mathrm{ind}(A), \mathrm{char} \ \widetilde{k})=1$. Additionally, if $A=L$, then assume the rank of the Hermitian form $h$ is at least $2$. If $X(F_v) \neq \emptyset$ for all $v \in V(F)$ discrete, then~${X(F) \neq \emptyset}.$ 
\end{thm}

\begin{proof} By Theorem \ref{7.3}, there exists a sequence of integers $0< n_1< \dots < n_r < \deg{A}$ such that $X \cong X_h(n_1, n_2, \dots, n_r).$ By \emph{loc.cit.} and Proposition~\ref{7.4}, we may assume $X=X_h(n_r).$ It suffices to show $X_h(n_r)(F) \neq \emptyset$.
Let $Z \subseteq C$ be a set of rigid points containing the ramification locus of $A$. For any $z \in Z$, let $V_z \subseteq C^{\mathrm{an}}$ be a strict affinoid neighborhood of $z$ such that $X_h(n_r)(\mathscr{M}(V_z)) \neq \emptyset$. Let $\mathscr{C}$ be a proper sncd model of $C^{\mathrm{al}}$ corresponding to a vertex set $S$ such that $\bigcup_{z \in Z} \partial{V_z} \subseteq S$ (Theorem~\ref{lordan}). Let $\pi_{\mathscr{C}}: C \rightarrow \mathscr{C}_s$ be the specialization morphism. If $x \in \bigcup_{z \in Z} V_z \cup S$, then $X(\mathscr{M}_x)\neq \emptyset$.

Let $x \in C \backslash (\bigcup_{z \in Z} V_z \cup S)$. Set $P_x:=\pi_{\mathscr{C}}(x)$ and $U_{\mathscr{C},x}:=\pi_{\mathscr{C}}^{-1}(P_x)$. By Theorem~\ref{bosh}, $\widehat{\mathcal{O}_{\mathscr{C}, P_x}}=\mathcal{O}^{\circ}(U_{\mathscr{C},x})$, so $\mathcal{O}^{\circ}(U_{\mathscr{C},x})$ is a complete regular local ring of dimension $2$. Let $\omega, \delta \in \mathcal{O}^{\circ}(U_{\mathscr{C},x})$ be generators of the maximal ideal as in  Corollary~\ref{4.4}. Set $F_{\mathscr{C},x}:=\mathrm{Frac} \ \mathcal{O}^{\circ}(U_{\mathscr{C},x})$.
Remark that $L \otimes_F F_{\mathscr{C},x} \cong F_{\mathscr{C},x}[\sqrt{\lambda}]$. For a fine enough model $\mathscr{C}$ (or, equivalently, large enough $S$), we may assume:
\begin{enumerate}
\item[(a)] up to multiplication by a square, $\lambda \in \mathcal{O}^{\circ}(U_{\mathscr{C},x})$ is either of the form $u$ or $u \omega$, where $u \in \mathcal{O}^{\circ}(U_{\mathscr{C},x})^{\times}$ (\emph{cf.} \cite[Lemma 4.2]{wu});
\item[(b)] the integral closure $R'$ of $\mathcal{O}^{\circ}(U_{\mathscr{C},x})$ in $F_{\mathscr{C},x}[\sqrt{\lambda}]$ has $\omega_1, \delta$ as generators of the maximal ideal, where $\omega_1:=\omega$ if $\lambda=u$ and $\omega_1:=\sqrt{u \omega}$ if $\lambda=u\omega$ (\emph{cf.} \cite[Sect.~9]{parsur});
\item[(c)]  $A$ is unramified on $R'$ except possibly at $(\omega_1)$ and $(\delta)$ (Corollary \ref{4.4}). 
\end{enumerate}
For any model $\mathscr{C}_1$ of $C^{\mathrm{al}}$ refining $\mathscr{C}$, $U_{{\mathscr{C}_1,x}} \subseteq U_{\mathscr{C},x}$, so $F_{\mathscr{C}_1,x} \subseteq F_{\mathscr{C},x}$.
As in \cite[Thm.~12.2]{parsur}, we prove by induction on $\mathrm{ind}(A \otimes_F F_{\mathscr{C},x})$ the existence of a proper sncd model $\mathscr{C}_1$ of $C^{\mathrm{al}}$ refining~$\mathscr{C}$ such that $X_h(n_r)(F_{\mathscr{C}_1,x}) \neq \emptyset$. 

Suppose $\mathscr{C}$ is a proper sncd model of $C^{\mathrm{al}}$ satisfying the conditions (a), (b), (c). If $\mathrm{ind}(A \otimes_F F_{\mathscr{C},x})=1$, $h$ corresponds to a quadratic form $q_h$ over $F_{\mathscr{C},x}$, and $h$ is isotropic if and only if $q$ is. One reduces to a question of isotropy of the quadratic form $q_h$ over $F_{\mathscr{C},x}$ (see proof of \cite[Thm. 12.2]{parsur}). Then, from \cite[Cor.~4.7]{hhk10}, $X_h(n_r)(F_{\mathscr{C},x}) \neq \emptyset$.
 Let $s \in \mathbb{N}$ be such that $s \geqslant 2$. Let $\mathscr{C}$ be a proper sncd model of $C^{\mathrm{al}}$ satisfying properties (a), (b), (c), and $\mathrm{ind}(A \otimes_F F_{\mathscr{C},x})< s$. Assume then that there exists a proper sncd model $\mathscr{C}_1$ refining~$\mathscr{C}$ such that $X_h(n_r)(F_{\mathscr{C}_1,x}) \neq \emptyset$. 
\begin{sloppypar}
Now let $\mathscr{C}$ be a proper sncd model of $C^{\mathrm{al}}$ satisfying conditions (a), (b), (c) and such that  $\mathrm{ind}(A \otimes_F F_{\mathscr{C},x})=s.$ Let~$D_x$ be the division algebra Brauer equivalent to $A \otimes_F F_{\mathscr{C},x}$ over $F_{\mathscr{C},x}$. Then ${\deg{D_x} \geqslant 2}.$ Let $(D_x, \tau_x, h_x)$ be the structure induced on~$D_x$ by ${(A \otimes_F F_{\mathscr{C},x}, \sigma, h)}$. By Theorem \ref{7.3}, there exists $X_{h_x}(m)/F_{\mathscr{C},x}$ homogeneous over $\textit{U}(D_x, \tau_x, h_x)$ such that ${X_{h_x}(m)(K) \neq \emptyset} \iff X_h(n_r)(K) \neq \emptyset$ for all extensions $K/F_{\mathscr{C},x}$. 
\end{sloppypar}
Let $\Lambda_x, \omega_x, \delta_x$ be as in Theorem \ref{7.4.1}. Let $<a_1, a_2, \dots, a_n>$ be a diagonal form of~$h_x$ with $a_i \in \Lambda_x$; then $\tau(a_i)=a_i$ for all $i$. Set $b_i:=Nrd_{D_x}(a_i) \in F_{\mathscr{C},x}^{\times}$. Let $\mathscr{C}_1$ be a proper sncd model refining $\mathscr{C}$ constructed via \cite[Lemma 4.3]{wu} for $b_1, b_2, \dots, b_n$. We may assume that it still satisfies conditions (a), (b), (c). By Theorem \ref{lordan}, $\mathscr{C}_1$ corresponds to a vertex set $S_1 \supseteq S$ of $C$. 
By construction, $\mathcal{O}^{\circ}(U_{\mathscr{C}_1,x})$ has generators $\omega_{1,x}, \delta_{1,x}$ of the maximal ideal such that $b_i=Nrd_{D_x}(a_i)=u_i \omega_{1,x}^{r_i} \delta_{1,x}^{s_i}$, $u_i \in \mathcal{O}^{\circ}(U_{\mathscr{C}_1,x})^{\times}$, $r_i, s_i \in \mathbb{Z}$, $i \in \{1,2,\dots, n\}.$  If $D_x \otimes_{F_{\mathscr{C},x}} F_{\mathscr{C}_1,x}$ is \emph{not} a division algebra, then $\mathrm{ind}(A \otimes_F F_{\mathscr{C}_1,x}) < \mathrm{ind}(A \otimes_F F_{\mathscr{C},x})$, so we may conclude by the inductive assumption that there exists a proper sncd model $\mathscr{C}_2$ refining~$\mathscr{C}_1$  such that $X_h(n_r)(F_{\mathscr{C}_2,x}) \neq \emptyset$.

If $D_x \otimes_{F_{\mathscr{C},x}} F_{\mathscr{C}_1,x}$ \emph{is} a division algebra, $h_x':=<a_1, a_2, \dots, a_n>$ is the Hermitian form induced by $h_x$, and now the $a_i$
 satisfy the properties of Theorem \ref{7.4.1}(2). Let $X_{h_x'}(m')/F_{\mathscr{C}_1,x}$ be the  homogeneous variety corresponding to $X_{h_x}(m) \times_{F_{\mathscr{C},x}} F_{\mathscr{C}_1,x}$.  Set $F_1:=\mathrm{Frac}(\widehat{\mathcal{O}^{\circ}(U_{{\mathscr{C}_1,x}})_{(\omega_{1,x})}})$ and $F_2:=\mathrm{Frac}(\widehat{\mathcal{O}^{\circ}(U_{\mathscr{C}_1,x})_{(\delta_{1,x})}}).$ These are both complete discretely valued fields containing~$F$, and by Remark~\ref{1.6bis},  $X_{h_x}(m)(F_j) \neq \emptyset$, implying $X_{h_x'}(m')(F_j) \neq \emptyset$ for $j=1,2.$
Using an induction argument (see proof of \cite[Thm. 12.2]{parsur}), the problem is reduced to one of isotropy of $h_x'$. By assumption, $h_x'$ is isotropic over~$F_j$, $j=1,2,$ so, by Theorem \ref{7.4.3}, it is isotropic over $F_{\mathscr{C}_1,x}$. Hence $X_{h_x'}(m')(F_{\mathscr{C}_1,x})\neq \emptyset$. Consequently, $X_{h_x}(m)(F_{\mathscr{C}_1,x}) \neq \emptyset$, so $X_{h}(n_r)(F_{\mathscr{C}_1,x}) \neq \emptyset$.

We have shown that there exists a proper sncd model $\mathscr{C}$ of $C^{\mathrm{al}}$ satisfying ${X_h(n_r)(F_{\mathscr{C},x}) \neq \emptyset}$. As $F_{\mathscr{C},x} \subseteq \mathscr{M}_x$, we obtain that $X_h(n_r)(\mathscr{M}_x) \neq \emptyset$. 
Hence,  for all $x \in C$, $X(\mathscr{M}_x) \neq \emptyset$, implying ${X_h(n_r)(F) \neq \emptyset}$ (\emph{cf.} \cite[Thm.~3.11]{une}).
\end{proof}
\subsection{Special unitary groups (\emph{cf.} \cite{wu})} \label{subsection7.2}
If we let $A$ be a central simple algebra with an involution $\sigma$ of the \emph{first} kind, and $h$ a Hermitian form on $(A, \sigma),$ then the group $\textit{SU}(A, \sigma, h)$ is a rational reductive linear algebraic group (\emph{cf.}  \cite[Lemma 5]{plati}). By using techniques similar to those in the case of unitary groups (the difference being that now ${L=F=Z(A)}$), one can prove the Hasse principle also holds for homogeneous spaces over special unitary groups. This was done in \cite{wu}. We refrain from translating this case to our setting as it would largely amount to repeating the arguments from Subsection \ref{unitary}.

\bigskip

{\footnotesize%
 \textsc{Vler\"e Mehmeti}, Sorbonne Université and Paris Cité, CNRS, IMJ-PRG, F-75005 Paris, France \par
  \textit{E-mail address}: \texttt{vlere.mehmeti@imj-prg.fr} 
}
\end{document}